\documentclass[11pt]{article}
\usepackage{graphicx} 
\usepackage{amsmath}
\usepackage{mathtools}
\usepackage{amsfonts}
\usepackage{amssymb}
\usepackage{amsthm}
\usepackage{comment}
\DeclarePairedDelimiter{\ceil}{\lceil}{\rceil}
\DeclarePairedDelimiter{\floor}{\lfloor}{\rfloor}
\DeclarePairedDelimiter{\brac}{ \{}{ \}}
\DeclarePairedDelimiter{\hk}{  (}{ )}

\usepackage{geometry}         
\usepackage{enumitem}         

\theoremstyle{plain}
\newtheorem{theorem}{Theorem}[section]
\newtheorem{lemma}[theorem]{Lemma}
\newtheorem{cor}[theorem]{Corollary}
\newtheorem{prop}[theorem]{Proposition}

\theoremstyle{definition}
\newtheorem{defn}[theorem]{Definition}
\newtheorem{example}[theorem]{Example}
\newtheorem{conj}{\textbf{Conjecture}}

\theoremstyle{remark}
\newtheorem{remark}[theorem]{Remark}

\title{\textbf{On the Fundamental Arithmetical Structure and Distribution of Lucky Numbers}}
\author{Marthinus Michael Dreeckmeier\\
\small School of Mathematics\\
\small University Of the Witwatersrand, Johannesburg\\
\small \texttt{michaeldreeck@gmail.com}
}
\date{\today}

\begin{document}

\maketitle

\begin{abstract}
    In this article, we will use elementary number theory techniques to investigate a sequence of integers defined by a sifting process called the lucky numbers. Ulam introduced lucky numbers as a sieve-based analogue of prime numbers. We derive an exact formula for the $n$th lucky number, providing a new tool for quantitative analysis. We formulate and prove a version of the Fundamental Theorem of Arithmetic for lucky numbers. This theorem provides an entirely new viewpoint on number theory. Building on the fundamental theorem, we introduce foundational definitions and analogues of arithmetical functions. Additionally, we prove an analogue of Bertrand's postulate for lucky numbers. Finally, we use the formula for the $n$th lucky number to prove a new result on the order of magnitude of the gaps between consecutive lucky numbers. We obtain an asymptotic bound that is much stronger than the best known bound for primes, and therefore obtain the first result that is known for lucky numbers but still conjectured to be true for primes. Together, these results establish a foundation for the arithmetic theory of lucky numbers and contribute to a broader understanding of sieve-generated sequences.
\end{abstract}

\section{Introduction}
Mathematicians have been studying prime numbers for thousands of years. Euclid first articulated the fundamental theorem of arithmetic, later proved in full generality by Gauss. The fundamental theorem of arithmetic is one of the main reasons why prime numbers are so special and have been the central object of study in number theory.
One of the earliest known methods for generating primes is the sieve of Eratosthenes, an ancient algorithm attributed to Eratosthenes of Cyrene.
Since prime numbers hold such great significance in mathematics and other areas like cryptography, it is in our interest to start looking at more general sequences of integers similar to prime numbers. One of these is called lucky numbers \cite{r1}. Lucky numbers are obtained by using a sieve similar to the sieve of Eratosthenes for prime numbers.

The process of finding lucky numbers is as follows, the first lucky number is defined as 1. Start with the set $A_0=\mathbb{N}$. Remove every second number from the set $A_0$, call this new set $A_1$. So $A_1=\{1,3,5,7,\dots\}$. Now, take the next remaining number in $A_1$, in this case it would be 3, and remove every third number in $A_1$, call this new set $A_2$. So $A_2=\{1,3,7,9,\dots\}$  Repeat this process and the remaining set, $A_\infty$, would be the set of all lucky numbers.  It follows that the set of lucky numbers is $$A_\infty=\{1, 3, 7, 9, 13, 15, 21, 25, 31, 33, 37, 43, 49, 51, 63,  \dots \}.$$
It is already known that lucky numbers have a distribution similar to prime numbers. This is a result of the lucky number theorem \cite{r2} - an analogue of the prime number theorem. In \cite{r2}, a recurrence relation is also given for lucky numbers.

In the book \textit{Unsolved Problems in Number Theory} by Richard K. Guy \cite{r4}, lucky numbers are discussed. Apart from the lucky number theorem, little is known about lucky numbers, but it is suggested that questions analogous to known results and conjectures about primes can be posed for lucky numbers.

The Riemann Hypothesis is arguably one of the most significant open problems related to the distribution of primes and is among the six remaining Millennium Prize Problems. Many other open questions about primes depend on the assumption that the Riemann Hypothesis is true. One such result, proven by Cramér under this assumption, is presented in his paper titled "On the order of magnitude of the difference between consecutive prime numbers" \cite{r5}. The result he obtained is the following. 
\begin{theorem}\label{thm7}
    Assuming that the Riemann Hypothesis is true, the following holds 
    $$p_{n+1}-p_n=O(\sqrt{p_n}\log p_n).$$ 
\end{theorem}
However, since the Riemann Hypothesis is still open, the result is yet to be proven.

 In Section \ref{s1}, we iterate the recurrence relation for lucky numbers to find an exact formula for the counting function of lucky numbers. This highlights the relationship between the recurrence relation and the counting function for lucky numbers. Building on this, we prove an exact formula for the $n$th lucky number, $l_n$.
 
In Section \ref{fta}, we formulate an analogue of the fundamental theorem of arithmetic for lucky numbers.  We will prove this theorem and introduce some new definitions, terminology and notation. Such a theorem could have significant implications, as many results in number theory, prime number theory, and practical fields like encryption rely on the fundamental theorem of arithmetic. This theorem provides a potential foundation for a new area of number theory, from which numerous questions and open problems naturally follow. This opens the possibility of studying number theory in settings where lucky numbers serve as the “building blocks” of the integers instead of prime numbers. In the final part of this section, we provide a way of representing integers in terms of lucky primes and give a generalization of the lucky prime conjecture \cite{r7}.

In section \ref{s2}, we prove an analogue of Bertrand's postulate for lucky numbers by using similar ideas to those used in the proof of the lucky number theorem. Although the lucky number theorem implies an analogue of Bertrand's postulate for sufficiently large $x$, we directly prove it for all $x\geq4$.

In Section \ref{s3}, we present our result on gaps between consecutive lucky numbers. This result is an analogue of a slightly stronger result than the conjectured result in Theorem \ref{thm7}. The focus of this section is to give a detailed proof of a new bound on the order of magnitude of the difference between consecutive lucky numbers.

\subsection*{Notation}
Throughout this report, $l$ is used to represent a lucky number. Similarly $p$ will be used to represent a prime. Define $l_n$ as the $n$th lucky number. Note that 1 is the first lucky number since it is the smallest number remaining after the sifting process is completed. To simplify calculations, we define $l_1=2$. Let $p_n$ be the $n$th prime number. Also note that the function $\log$ will be used as the natural logarithm. In addition, the symbol $\sim$ will be used to indicate asymptotic equality as the parameter of interest increases without bound.
\begin{defn}
    Define the function $L :[1,\infty)\rightarrow \mathbb{N}$ as 
    $$L(x) := \vert \{l\,|\, l\leq x, \text{ and }l \text{ is a lucky number}  \} \vert.$$ 
    So $L(x)$ is the number of lucky numbers less than or equal to $x$. Thus, $L(x)$ will be referred to as the counting function of lucky numbers.
\end{defn}

\section{Formulas for Lucky Numbers}\label{s1}
In the proof of the lucky number theorem, Hawkins and Briggs \cite{r2} obtained the following recurrence relation,
\begin{equation}
\label{q2}
    R(n,x)=R(n-1,x)-\floor*{\frac{R(n-1,x)}{l_{n-1}}}, \text{ for } n\geq 2,
\end{equation}
where $$R(n,x) =\vert\{k\,|\,k\in A_{n-1}, k\leq x\}\vert.$$
Iterating this recurrence relation, we get a formula for $L(x)$.
\begin{prop}\label{prop1}
    Take any $x\geq1$. Note that $l_{L(L(x))}$ is the largest lucky number less than or equal to $L(x)$. Then the following formula holds for $L(x)$.
    \begin{equation}
        L(x)=\ceil*{\dotsi \ceil*{\ceil*{\ceil*{\floor{x}\hk*{1-\frac{1}{l_1}}}\left(1-\frac{1}{l_2}\right)}\left(1-\frac{1}{l_3}\right)}\dotsi \left(1-\frac{1}{l_{L(L(x))}}\right)}.
    \end{equation}
\end{prop}
\begin{proof}
    First note that there are $L(L(x))$ lucky numbers less than or equal to $L(x)$, therefore it follows that $l_{L(L(x))}$ is the largest lucky number less than or equal to $L(x).$
      Consider the defined sequence $R(n,x)$, where $R(1,x)=\floor{x}$. Note that by definition, $R(i,x)\in \mathbb{N}_0$ for all $i$. In the first step of the sifting process, $\floor*{\frac{R(1,x)}{2}}$ numbers will be eliminated (this follows since every second integer will be eliminated). Hence, it follows that 
      $$R(2,x)=R(1,x)-\floor*{\frac{R(1,x)}{2}}=\ceil*{\frac{R(1,x)}{2}},$$
      since $R(i,x)$ is an integer and $R(i,x)-\floor*{\frac{R(i,x)}{2}}=R(i,x)+\ceil*{-\frac{R(i,x)}{2}}=\ceil*{R(i,x)-\frac{R(i,x)}{2}}=\ceil*{R(i,x)\left(1-\frac{1}{2}\right)}$ for all $i$.
      Similarly, in the second step of the sifting process $\floor*{\frac{R(1,x)}{l_2}}$ numbers will be eliminated. Hence,
      $$R(3,x)=R(2,x)-\floor*{\frac{R(2,x)}{l_2}}=\ceil*{R(2,x)\left(1-\frac{1}{l_2}\right)}.$$
      Repeat this process to get the result 
      \begin{equation}\label{eq2}
          R(i+1,x)=R(i,x)-\floor*{\frac{R(i,x)}{l_{i+1}}}=\ceil*{R(i,x)\left(1-\frac{1}{l_{i+1}}\right)},  \forall \, i.
      \end{equation}
(This result is similar to the recurrence relation obtained in equation (\ref{q2}) by \cite{r2}.)

Note that if $l_{i+1}>R(i,x)$ for some $R(i,x)$, then $$\ceil*{R(i,x)(1-1/l_{i+1})}=\ceil*{R(i,x)-R(i,x)/l_{i+1}}=R(i,x)$$ (since $R(i,x)\in \mathbb{N}_0$ and $\frac{R(i,x)}{l_{i+1}}<1$). Also note that this would imply that $R(i,x)=L(x)$ for any $i\geq {L(L(x))}$. Hence, from this, it is only necessary to use the lucky numbers in the sifting process up to the largest $l_{{L(L(x))}}$ such that $l_{L(L(x))}\leq L(x)$. Intuitively, this implies that if there are fewer than $l_{{L(L(x))}+1}$ numbers remaining less than or equal to $x$, then there is no $l_{{L(L(x))}+1}$th number that can be eliminated. Hence, it follows that 
\begin{equation} \label{eq3.3}
    L(x)=R({L(L(x))},x).
\end{equation}
Using the recurrence relation obtained in equation (\ref{eq2}), 
\begin{align*}
    L(x) & =R({L(L(x))},x)=\ceil*{R({L(L(x))}-1,x)\left(1-\frac{1}{l_{{L(L(x))}}}\right)}\\
    & =\ceil*{\ceil*{R({L(L(x))}-2,x)\left(1-\frac{1}{l_{{L(L(x))}-1}}\right)}\left(1-\frac{1}{l_{{L(L(x))}}}\right)}\\
    & \;\;\vdots\\
    &=\ceil*{\dotsi \ceil*{\ceil*{\ceil*{R(1,x)\hk*{1-\frac{1}{l_1}}}\left(1-\frac{1}{l_2}\right)}\left(1-\frac{1}{l_3}\right)}\dotsi \left(1-\frac{1}{l_{L(L(x))}}\right)}\\
    &=\ceil*{\dotsi \ceil*{\ceil*{\ceil*{\floor{x}\hk*{1-\frac{1}{l_1}}}\left(1-\frac{1}{l_2}\right)}\left(1-\frac{1}{l_3}\right)}\dotsi \left(1-\frac{1}{l_{L(L(x))}}\right)}.
\end{align*}
\end{proof}
\begin{example}
    For example, when this formula is applied to $x=10$,we have, $\ceil*{10/2}=5$, $\ceil*{5(2/3)}=4$, and then the next lucky number is 7, but $7>4$, so the process ends. It follows that $L(10)=4$, and the corresponding largest lucky number less than 4 is $l_{L(L(x))} = 3$.
\end{example}
 Also note that equation (\ref{eq3.3}) gives the relationship between $L(x)$ and $R(i,x)$.

The following proposition relates $l_n$ to the lucky numbers less than or equal to $n$. This gives a useful method to calculate $l_n$ without iterating the entire lucky sieve.
\begin{prop}\label{prop2}
        Let $n\geq2$. Let $l_\beta$ be the largest lucky number less than or equal to $n$. Then 
    \begin{align}
        &l_n=\ceil*{\ceil*{\ceil*{ \dotsi \ceil*{\ceil*{n\left(\frac{l_\beta}{l_\beta -1}\right)-1}\left(\frac{l_{\beta-1}}{l_{\beta-1} -1}\right)-1} \dotsi }\left(\frac{l_{2}}{l_{2} -1}\right)-1}\hk*{\frac{l_1}{l_1-1}}-1}\label{eq3.10}\\
        &l_n=\floor*{\left(\floor*{\left(\floor*{ \dotsi  \floor*{\left( \floor*{(n l_\beta -1)/(l_\beta -1)}l_{\beta-1}-1\right)/(l_{\beta-1}-1)} \dotsi }l_2-1\right)/(l_2-1)}\right)l_1-1)/(l_1-1)}.\label{eq3.11}
    \end{align}
   
\end{prop}
Erdös and Jabotinsky \cite{r3} obtained a similar formula for sequences of integers defined by sieves. Note that, as defined in Proposition \ref{prop2}, we have $\beta=L(n)$.
\begin{proof}
    Note that the two formulas (\ref{eq3.10}) and (\ref{eq3.11}) are equivalent. This is a result of basic manipulations of the floor and ceiling functions and the formula  for $a,b \in \mathbb{N}, \floor*{\frac{a}{b}}=\ceil*{\frac{a+1}{b}}-1$ or equivalently $\floor*{\frac{a-1}{b}}=\ceil*{\frac{a}{b}}-1$. By applying this to each ceiling function in (\ref{eq3.10}), the formula in (\ref{eq3.11}) is obtained. Therefore, it suffices to prove only one of the formulas.
    
    Proof for (\ref{eq3.10}). First, define the preimage of the counting function 
    $$L^{-1}(n) =\{k\in \mathbb{N}\,|\,L(k)=n\}.$$ 
   Note that
    $$L(l_n)=L(l_n+1)=\cdots=L(l_{n+1}-1)=n.$$
    Therefore, it follows that $\min{(L^{-1}(n))}=l_n$. By using this fact, it is required to show that the formula in (\ref{eq3.10}) maps $n$ to the minimum of the preimage $L^{-1}(n)$.  
    First we construct a finite sequence of sets. Let 
    \begin{align*}
        &B_1:=\{y\in \mathbb{N}\,|\, \ceil*{y\hk*{1-\frac{1}{l_{\beta}}}}=n \}\\
        &B_2:=\{y\in \mathbb{N}\,|\, \ceil*{y\hk*{1-\frac{1}{l_{\beta-1}}}}= \min(B_1) \}\\
        &B_3:=\{y\in \mathbb{N}\,|\, \ceil*{y\hk*{1-\frac{1}{l_{\beta-2}}}}= \min(B_2) \}\\
        &\;\;\vdots\\
        &B_{\beta}:=\{y\in \mathbb{N}\,|\, \ceil*{y\hk*{1-\frac{1}{l_{1}}}}= \min(B_{\beta-1}) \}.
    \end{align*}
    Noting that $L(l_n)=n$ and in Proposition \ref{prop1}, $l_{L(L(l_n))}$ is the largest lucky number less that or equal to $L(l_n)$, we have $l_\beta=l_{L(n)}=l_{L(L(l_n))}$. Therefore by construction, and Proposition \ref{prop1}, we have $L(B_{\beta})=n$ and 
    $$\min{(L^{-1}(n))}=\min{(B_{\beta})}=l_n.$$
    With this information, we only need to show that the first step of the formula in (\ref{eq3.10}), maps to the minimum value of $B_1$, i.e. we will show that $\ceil*{n\left(\frac{l_\beta}{l_\beta -1}\right)-1}=\min{(B_1)}$. Then we show that the second step maps to $\min{(B_2)}$, i.e. $\ceil*{\min{(B_1)}\left(\frac{l_\beta}{l_\beta -1}\right)-1}=\min{(B_2)}.$ In general, we show that $\ceil*{\min{(B_{i})}\left(\frac{l_\beta}{l_\beta -1}\right)-1}=\min{(B_{i+1})}$ for the $i$th step.

 Consider $y \in B_{\beta-i+1}$. This would be any $y\in\mathbb{N}$ such that $\ceil*{y\hk*{1-\frac{1}{l_i}}}=\min(B_{\beta-i})$. Note that based on the definition in Proposition \ref{prop1}, this would imply that $y\geq l_i$. It is possible to write $y=ql_i+r$ for some quotient $q$ and some remainder $0\leq r<l_i$. Use this to simplify the expression for the $i$th step to get 
    \begin{equation} \label{eq36}
        \ceil*{y\hk*{1-\frac{1}{l_i}}}=y-\floor*{\frac{y}{l_i}}=y-\floor*{\frac{ql_i+r}{l_i}}=y-q-\floor*{\frac{r}{l_i}}=y-q.
    \end{equation}
    We have
    $$\ceil*{y\hk*{1-\frac{1}{l_i}}} \neq \ceil*{(y+2)\hk*{1-\frac{1}{l_i}}}.$$
    Since
    $$\ceil*{(y+2)\hk*{1-\frac{1}{l_i}}}=y+2-\floor*{\frac{y+2}{l_i}}=y+2-\floor*{\frac{ql_i+r+2}{l_i}}=y+2-q-\floor*{\frac{r+2}{l_i}}>y-q,$$
    this is true since, $l_i\geq 2$ and $r<l_i$, and it follows that $r+2<2l_i \implies \frac{r+2}{l_i}<2$, hence $\floor*{\frac{r+2}{l_i}}\leq1$. Therefore, we can conclude that if $y\in B_{\beta-i+1}$, then $y+2\notin B_{\beta-i+1}$. Without loss of generality, this is equivalent to saying, if $y\in B_{\beta-i+1}$, then $y-2\notin B_{\beta-i+1}$. We can ask the question, if $y\in B_{\beta-i+1}$, then when is $y-1\in B_{\beta-i+1}$? This is relevant because it is needed to show that in the case where both $y, y-1\in B_{\beta-i+1}$, then the $i$th step of formula (\ref{eq3.10}) only maps to $y-1=\min{(B_{\beta-i+1})}$.

    Using a similar method that we used to simplify $\ceil*{(y+2)\hk*{1-\frac{1}{l_i}}}$, we can show that 
    $$\ceil*{(y-1)\hk*{1-\frac{1}{l_i}}}=y-1-q-\floor*{\frac{r-1}{l_i}}.$$
    From this, it is easy to see that $\ceil*{y\hk*{1-\frac{1}{l_i}}}=\ceil*{(y-1)\hk*{1-\frac{1}{l_i}}}$ only if $r=0$. Thus we get the following:
    $$
    \min{(B_{\beta-i+1})}=\begin{cases}
        y, &\text{if } r\neq 0;\\
        y-1, &\text{if } r=0.
    \end{cases}
    $$
   Consider the two cases.\\
    Case 1: 
    when $r\neq 0$. Then 
    \begin{align*}
        \ceil*{\ceil*{y\hk*{1-\frac{1}{l_i}}}\hk*{\frac{l_i}{l_i-1}}-1} = \ceil*{\hk*{\frac{(y-q)l_i}{l_i-1}}-1} =\ceil*{\hk*{\frac{(yl_i-ql_i}{l_i-1}}-1}
    \end{align*}
    which follows by substituting equation (\ref{eq36}).
    Using the fact that $y=ql_i+r, \therefore ql_i=y-r$,
    \begin{align*}
        &\ceil*{\hk*{\frac{(yl_i-ql_i}{l_i-1}}-1}=\ceil*{\hk*{\frac{(yl_i-y+r}{l_i-1}}-1}\\
        =&\ceil*{\hk*{\frac{(y(l_i-1)+r}{l_i-1}}-1}=y-1+\ceil*{\frac{r}{l_i-1}} =y=\min{(B_{\beta-i+1})}.
    \end{align*}
    (Since $r>0, \ceil*{\frac{r}{l_i-1}}=1$).\\
   Case 2: 
    when $r=0$. Then the preimage of the $i$th step is $\{y-1,y\}$. Hence, the smallest is $y-1$, so the formula in (\ref{eq3.10}) should map to $y-1$. We have 
    $$
    \ceil*{\ceil*{y\hk*{1-\frac{1}{l_i}}}\hk*{\frac{l_i}{l_i-1}}-1} = \ceil*{\hk*{\frac{(y-q)l_i}{l_i-1}}-1} =\ceil*{\hk*{\frac{(yl_i-ql_i}{l_i-1}}-1}.
    $$
    Since $r=0, y=ql_i$, we get
    $$
    \ceil*{\hk*{\frac{(yl_i-ql_i}{l_i-1}}-1}=\ceil*{\hk*{\frac{(yl_i-y}{l_i-1}}-1}=\ceil*{\hk*{\frac{(y(l_i-1)}{l_i-1}}-1}=y-1=\min{(B_{\beta-i+1})}.
    $$
    This concludes the proof.
\end{proof}
\begin{example}
    For example, using the formula in (\ref{eq3.10}), we can calculate the third lucky number, $\ceil*{3(3/2)-1}=\ceil*{3.5}=4$. Next, $\ceil*{4(2/1)-1}=7=l_3$. We can also use formula (\ref{eq3.11}) to find the 8th lucky number, e.g. $\floor*{(8\times7-1)/6)}=\floor*{9+1/6}=9$, $\floor*{(9\times3-1)/2)}=13$, finally $\floor*{(13\times2-1)/1)}=25=l_8.$
\end{example}

\section{The Fundamental Theorem of Arithmetic for Lucky Numbers} \label{fta}
In this section, we aim to develop an analogue of the fundamental theorem of arithmetic for lucky numbers and explore its applications. 

An intuitive approach to get this result follows from the fact that every number is either eliminated by a lucky number (or the number 2) or is a lucky number itself. Keeping this in mind, we can define a new binary operator to act as \lq\lq multiplication\rq\rq\,for lucky numbers. 
\begin{defn} \label{def4.1}
    Let $\mathbb{L}=(A_\infty \backslash\{1\}) \cup \{2\}=\{2,3,7,9,13,15,\cdots\}$. Define the binary operator $\ast  : \mathbb{L}\times\mathbb{N}\to \mathbb{N}$ so that $l_{i}*n$ is equal to the $n$th number eliminated by $l_i$ in the sifting process. Thus, $l_{i}*n$ is the $(l_in)$th number in the set $A_{i-1}$.
\end{defn}
\begin{example}
    For example, $2*3=6$ since 6 is the third number eliminated by 2 and $3*1=5$ since 5 is the first number eliminated by 3.
\end{example}
Note that the binary operator, $*$, is not commutative, for example, the third number eliminated by 7 cannot be equal to the seventh number eliminated by 3, since it has already been eliminated. Additionally, the binary operator, $*$, is also not associative, for example, $(2*2)*1$ is not defined, however $2*(2*1)=4$ falls well within the requirements of the definition.

This gives us enough information to obtain an analogue of the fundamental theorem of arithmetic by using the definition for $*$.
\begin{theorem}\label{thmA}
   Every positive integer $n\geq1$ can be uniquely written in the following way 
   $$n=l_{\alpha_1}*(l_{\alpha_2}*(l_{\alpha_3}*(l_{\alpha_4}*(\cdots*l_{\alpha_k})\cdots);$$
   where $l_{\alpha_i}\in\mathbb{L} $ for all $i<k$ and $l_{\alpha_k}\in A_{\infty}$.
\end{theorem}
\begin{remark}
  Note that if $n=l_{\alpha_1}*(l_{\alpha_2}*(l_{\alpha_3}*(l_{\alpha_4}*(\cdots*l_{\alpha_k})\cdots)$, it is possible that $l_{\alpha_k}=1$. This is because it need not be the case that $l*1=l$, for example $3*1=5$. Clearly $l_{\alpha_i}\neq 1$ for all $i<k$ since then the expression would not be defined.   
\end{remark}
\begin{proof} 
    We will prove this theorem using induction. Clearly, all numbers in $ A_{\infty}$ can be written uniquely as themselves. Assume that all numbers less than $n$ can be written uniquely in the form as in the theorem. If $n\in  A_{\infty}$, then we are done. If not, then it is possible to write $n=l*n_1$ for some $l\in \mathbb{L}$ and $n_1\in \mathbb{N}$. This is because $n$ is not a lucky number, so $n$ must be eliminated by some lucky number in the sifting process. This representation is also unique because $n$ can only be eliminated once. Note that $l*n_1\geq l(n_1)$. This is because $l*n_1$ is the $n_1$th multiple of $l$ in $A_k$ for some $k$ and $A_k \subseteq A_0=\mathbb{N}$. We therefore have $n=l*n_1\geq l(n_1)$ and since $l\geq 2$, it follows that $n_1<n$. By the induction hypothesis, $n_1$ can be factored uniquely and therefore $n$ can be factored uniquely. The result follows by induction.
\end{proof} 

To write an integer $n$ in the form as in the fundamental theorem of arithmetic for lucky numbers, use the following process: if $n$ is in $ A_{\infty}$, it is already in the correct form. If not, then $n=l_{\alpha_1}*a_1$ for some $l_{\alpha_1} \in \mathbb{L}$ and some $a_1 \in \mathbb{N}$. Next, if $a_1\in  A_{\infty}$, the process ends, if not, then $a_1=l_{\alpha_2}*a_2$ or some $l_{\alpha_2} \in \mathbb{L}$ and some $a_2 \in \mathbb{N}$. Hence $n=l_{\alpha_1}*(l_{\alpha_2}*a_2)$. If $a_2\in  A_{\infty}$, the process ends, if not, continue this process until $a_k\in  A_{\infty}$ for some $k$.

Theorem \ref{thmA} makes it possible to define what a lucky divisor of an integer would be.
\begin{defn}
    We call $l\in  A_{\infty}\cup\{2\}$ a lucky divisor of $n$ with respect to $*$, if $n=l_{\alpha_1}*(l_{\alpha_2}*(l_{\alpha_3}*(l_{\alpha_4}*(\cdots*l_{\alpha_k})\cdots)$ and $l\in \{l_{\alpha_1},l_{\alpha_2},\dots,l_{\alpha_k}\}$. If $l$ is a lucky divisor of $n$, we denote this by $l \overset{*}{|} n$. 
\end{defn}
\subsubsection*{How to factor an integer into its lucky factors}
As a direct consequence of the proof of Proposition \ref{prop1}, the following algorithm can be used to factor any integer $n$. Check to see if $n \in  A_{\infty}$. If this is the case, we are done, if not, check to see if $2|n$. If this holds, then we know that $n=2*n_1$ for some $n_1=n/2$ less than $n$ and then restart the same algorithm to factor $n_1$. If not, calculate $R(2,n)=\ceil*{n\hk*{1-\frac{1}{l_1}}}$. If $l_2|R(2,n)$, we can write $n=l_2*n_1$ where $n_1=R(2,n)/l_2$ and restart the algorithm to factor $n_1$. If not, calculate $R(3,n)=\ceil*{R(2,n)\hk*{1-\frac{1}{l_2}}}$. If $l_3|R(3,n)$, we can write $n=l_3*n_1$ where $n_1=R(3,n)/l_3$. If not, continue this process up to $l_{L(L(x))}$.
\begin{example}
We can factor 22 using this algorithm. First, we get $22=2*11$. Next, we need to factor 11. We have $\ceil{11(1-1/2)}=6$ and $6=2(3)$, hence $11=3*2$. Finally, we have $22=2*(3*2)$. 
\end{example}

Using similar ideas, we can find a formula for $l_i*n$.
\begin{prop}
    For any $l_i \in \mathbb{L}$ and $n\in \mathbb{N}$ we have the following
    $$l_i*n=\ceil*{\ceil*{\ceil*{ \dotsi \ceil*{\ceil*{(l_in)\left(\frac{l_{i-1}}{l_{i-1} -1}\right)-1}\left(\frac{l_{i-2}}{l_{i-2} -1}\right)-1} \dotsi }\left(\frac{l_{2}}{l_{2} -1}\right)-1}\hk*{\frac{l_1}{l_1-1}}-1}$$
\end{prop}
The proof of this proposition can be done in a similar way to the proof of Proposition \ref{prop2}.

 Note that it is possible that the same lucky number appears twice or more in different stages of the expression, for example, $2*2=4$. This gives us another definition inspired by an existing one for primes.
\begin{defn}
    The number of times that a lucky number $l$
    appears in the expression of $n$ as given by Theorem \ref{thmA}, will be called the l-adic valuation od n or the l-adic order of $n$, and shall be denoted by ${\overset{*}{\nu}}_l(n)$. This can also be referred to as the highest \lq power \rq\,\,of $l$ in $n$.
\end{defn}

Theorem \ref{thmA} has many consequences. One of the consequences is that it is possible to define new arithmetical functions similar to the already existing arithmetical functions for primes. 

\begin{defn}
    Define the lucky $\overset{*}{\omega} $ and $\overset{*}{\Omega}$ functions, by
    \begin{align*}
        \overset{*}{\omega}(n)&:=\sum_{l\overset{*}{|}n}{1},\\
        \overset{*}{\Omega}(n)&:=\sum_{l \overset{*}{|}n}{{\overset{*}{\nu}}_l(n)}.
    \end{align*}
\end{defn}
The first function counts the number of distinct lucky divisors of $n$ and the second counts the lucky divisors with their order.
This is similar to the prime omega functions $\omega$ and $\Omega$. Recall that the prime omega functions are defined as
\begin{align*}
    \omega (n)&=\sum_{p|n}{1},\\
    \Omega (n)&=\sum_{p|n}{\nu_p(n)}.
\end{align*}
A natural question is therefore to ask how do these new functions relate to the original functions defined for primes and what are the average orders of these new functions? For the original prime omega functions, we know
$$\sum_{n\leq x}{\Omega(n)}\sim x\log\log x$$
and
$$\sum_{n\leq x}{{\omega(n)}^k}\sim x^k\log\log x, \,k\in \mathbb{N}.$$
\begin{conj}
  We have
$$
\sum_{n\leq x}{\overset{*}{\omega}(n)}\sim Ax\log \log x,
$$
$$
\sum_{n\leq x}{\overset{*}{\Omega}(n)} \sim Bx\log \log x,
$$
for some constants $A,B\in \mathbb{R}.$ 
If this is the case, is $A=B$?
\end{conj}
Using simple ideas, we easily obtain the following lemma.
\begin{lemma}
For $x\geq2$, we have
    $$\sum_{n\leq x}{\overset{*}{\Omega}(n)}=O(x\log{x}).$$
\end{lemma}
\begin{proof}
    First, we prove that $\overset{*}{\Omega}(n)\leq c\log{n}+1 \text{ for some constant } c$ and any integer $n\geq1$.
    For any integer $m\geq 2$, we have $2*m=2m\geq2*2$. Furthermore, for any $l_i\in \mathbb{L}$, we have $l_i*2\geq2l_i\geq2*2$. This was explained in the proof of Theorem \ref{thmA}. Combining these two facts, it follows that if $\overset{*}{\Omega}(n)=k$, then $n\geq \underbrace{2*2*2*\cdots*2*1}_{k\text{ terms}}=2^{k-1}$. Therefore, it follows that $\log_2{n}\geq\log_2{2^{k-1}}=k-1$. Substitute $\overset{*}{\Omega}(n)=k$ and the result follows. 
   Thus, it follows that
    \begin{align*}
        \sum_{n\leq x}{\overset{*}{\Omega}(n)}&\leq c\sum_{n\leq x}{1+\log{n}}\\
        &=c(\floor{x}+\log{\floor{x}!})
    \end{align*}
    Using the fact that $\log{\floor{x}!}=O(x\log{x})\text{, for } x\geq 2$ by \cite{r6}, the result follows.
\end{proof}

\begin{remark}
    We have the following potentially useful expansion
    \begin{align*}
        \sum_{n\leq x}{\overset{*}{\Omega}(n)}=&\smashoperator[r]{\sum_{n\leq \floor*{\frac{x}{l_1}}}}{(1+\overset{*}{\Omega}(n))}+\smashoperator[r]{\sum_{n\leq \floor*{\ceil*{x\hk*{\frac{l_1-1}{l_1}}}\frac{1}{l_2}}}}{(1+\overset{*}{\Omega}(n))}+\smashoperator[r]{\sum_{n\leq \floor*{\ceil*{\ceil*{x\hk*{\frac{l_1-1}{l_1}}}\frac{l2-1}{l_2}}\frac{1}{l_3}}}}{(1+\overset{*}{\Omega}(n))}\\
        &+\dots+\smashoperator[r]{\sum\limits_{\substack{n\leq \floor*{\dotsi \ceil*{\ceil*{\ceil*{x\hk*{\frac{l_1-1}{l_1}}}\left(\frac{l_2-1}{l_2}\right)}\left(\frac{l_3-1}{l_3}\right)}\dotsi \left(\frac{1}{l_{L(L(x))}}\right)}}}}{(1+\overset{*}{\Omega}(n))} \qquad\qquad +L(x).
    \end{align*}
    This result can easily be derived using the same logic as in the subsection on how to factor integers with respect to the defined binary operator. The first summation is for integers of the form $l_1*m$ less than $x$, the second is for integers of the form $l_2*m$ less than $x$ and so on. The only numbers that remain and are not of the form $l*m$ are the lucky numbers themselves, hence the term $L(x)$.
\end{remark}

The next objective is to extend the binary operator $*: \mathbb{L}\times\mathbb{N}\to \mathbb{N}$ to some larger set acting on the left of the operator. This will allow us to define what a divisor of an integer would be and lead to more possible arithmetical functions and questions.
\begin{defn}
    Let $B=\{n\in\mathbb{N} \,| \,1\overset{*}{\nmid}n\}$ and $C=\{n\in\mathbb{N} \,| \,n=l_{\alpha_1}*(l_{\alpha_2}*(\cdots*(2*1)\cdots)\}$. Let $S=B\cup C$. Let $a\in S$, $b \in \mathbb{N}$. Then by Theorem \ref{thmA} we can write $a=l_{\alpha_1}*(l_{\alpha_2}*(l_{\alpha_3}*(\cdots*l_{\alpha_{k}})\cdots)$, where $l_{\alpha_{k}}\neq 1$ if $a\in B$ or $a=l_{\alpha_1}*(l_{\alpha_2}*(l_{\alpha_3}*(\cdots*2)\cdots)$ if $a\in C$ and then let $l_{\alpha_{k}}=2$. Therefore, $l_{\alpha_{k}} \in \mathbb{L}$. Define the extension $\circ :S\times \mathbb{N}\to \mathbb{N}$ as follows, 
    $$a\circ b:=l_{\alpha_1}*(l_{\alpha_2}*(l_{\alpha_3}*(l_{\alpha_4}*(\cdots*(l_{\alpha_{k}}*b)\cdots).$$
\end{defn}
\begin{remark}
    This extension is associative, that is, for $a,b\in S$ and $c\in \mathbb{N}$ we have $a\circ (b\circ c)=(a\circ b)\circ c$.
\end{remark}

    To see why the extension is associative, we can write $a=l_{\alpha_1}*(l_{\alpha_2}*(l_{\alpha_3}*(\cdots*l_{\alpha_{k_1}})\cdots)$, $b=l_{\beta_1}*(l_{\beta_2}*(l_{\beta_3}*(\cdots*l_{\beta_{k_2}})\cdots)$. Then, by definition, $a\circ (b\circ c)=a\circ (l_{\beta_1}*(l_{\beta_2}*(\cdots*(l_{\beta_{k_2}}*c)\cdots))=l_{\alpha_1}*(l_{\alpha_2}*(\cdots*(l_{\alpha_{k_1}}*(l_{\beta_1}*(l_{\beta_2}*(\cdots*(l_{\beta_{k_2}}*c)\cdots)\cdots)$. We consider $(a\circ b)\circ c=(l_{\alpha_1}*(l_{\alpha_2}*(\cdots*(l_{\alpha_{k_1}}*(l_{\beta_1}*(l_{\beta_2}*(\cdots*(l_{\beta_{k_2}})\cdots)\cdots))\circ c=l_{\alpha_1}*(l_{\alpha_2}*(\cdots*(l_{\alpha_{k_1}}*(l_{\beta_1}*(l_{\beta_2}*(\cdots*(l_{\beta_{k_2}}*c)\cdots)\cdots)=a\circ (b\circ c)$ using the definition.
    
\begin{remark}
    This extension does not affect the unique factorization in Theorem \ref{thmA}. Since the extension is associative, it is possible to write 
$$n=l_{\alpha_1}*(l_{\alpha_2}*(l_{\alpha_3}*(l_{\alpha_4}*(\cdots*l_{\alpha_{k}})\cdots)=l_{\alpha_1}\circ l_{\alpha_2}\circ l_{\alpha_3}\circ l_{\alpha_4}\circ \cdots \circ l_{\alpha_{k}}.$$
Also note that for example in the case of $2\circ n$, it would be undefined to write $2\circ1\circ n$ since $1\circ n$ is not defined. 
\end{remark}

Extending the operator to $\mathbb{N}\times\mathbb{N}$ is not possible if we want the uniqueness in Theorem \ref{thmA} to hold. This is because there is no clear way to define what $1\circ n$ should be. For example, if we decide to choose $1\circ n=n$, then $5\circ n=3\circ 1\circ n=3\circ n$ which changes the uniqueness in Theorem \ref{thmA}.

With this extension, we can introduce even more definitions.
\begin{defn}
    Let $d,n\in \mathbb{N}$. We say that $d$ is a divisor of $n$ with respect to the binary operator $\circ $, denoted by $d\overset{\circ }{|}n$, if $n=l_{\alpha_1}\circ l_{\alpha_2}\circ l_{\alpha_3}\circ l_{\alpha_4}\circ \cdots\circ l_{\alpha_{k_1}}$, $d=l_{\beta_1}\circ l_{\beta_2}\circ l_{\beta_3}\circ l_{\beta_4}\circ \cdots\circ l_{\beta_{k_2}}$ and the tuple $(l_{\beta_1},l_{\beta_2},\dots,l_{\beta_{k_2}})$ is contained in the tuple $(l_{\alpha_1},l_{\alpha_2},\dots,l_{\alpha_{k_1}})$.
\end{defn}
\begin{example}
    We have $10=2*(3*1)=2\circ3\circ1$, then the divisors of 10 are $10=2\circ3\circ1,\, 6=2\circ3,\,5=3\circ1,\,2,\,3,\,1$.
\end{example}
Since we defined what the divisors of an integer are, we can define the lucky divisor functions similar to the divisor functions for primes.
\begin{defn}
    Let $n\in \mathbb{N}$. We define the $i$th power lucky divisor function as follows
    $$
    \overset{\circ }{\sigma}_i(n):=\sum_{d\overset{\circ }{|}n}{d^i}.
    $$
    For $i=0$, this can also be denoted as $\overset{\circ }{d}(n)$ and would give the number of divisors of $n$ with respect to $\circ $.
\end{defn}
\begin{remark}
    We can change the notation for the lucky omega function to be consistent with this extension i.e. 
\begin{align*}
    \overset{*}{\omega}(n)=\overset{\circ}{\omega}(n) & \text{ and } \overset{*}{\Omega}(n)=\overset{\circ}{\Omega}(n).
\end{align*}

\end{remark}
Using simple techniques, we can deduce the following lemma.
\begin{lemma}
    For any integer $n\in \mathbb{N}$, we have the following bounds:
    $$
    \frac{(\overset{\circ }{\omega}(n))^2+\overset{\circ }{\omega}(n)}{2}\leq \, \overset{\circ }{d}(n)\leq \frac{(\overset{\circ }{\Omega}(n))^2-\overset{\circ }{\Omega}(n)}{2}+\overset{\circ }{\omega}(n).
    $$
    With equality occurring if $\overset{\circ }{\omega}(n)=\overset{\circ }{\Omega}(n)$.
\end{lemma}
\begin{proof}
Let $n=l_{\alpha_1}\circ (l_{\alpha_2}\circ (l_{\alpha_3}\circ (\cdots\circ l_{\alpha_{k}})\cdots)$. Then $ \overset{\circ }{\Omega}(n)=k$, and $n$ has the following corresponding $k-$tuple $(l_{\alpha_1},l_{\alpha_2},\dots,l_{\alpha_{k}})$. We will denote this tuple by $u(n)$. The tuple, $u(n)$, contains at most $\overset{\circ }{\omega}(n)$ distinct 1-tuples. Next, $u(n)$ contains at most $k-1$ distinct 2-tuples, at most $k-2$ distinct 3-tuples, $\dots$, at most $k-i+1$ distinct $i-$tuples and $1$ $k-$tuple. Summing the first $k-1$ natural numbers and adding $\overset{\circ }{\omega}(n)$ will give the upper bound as required.

For the lower bound, construct a new tuple of distinct elements (lucky numbers) using the following algorithm. Take the first element in the $k-$tuple $u(n)=(l_{\alpha_1},l_{\alpha_2},\dots,l_{\alpha_{k}})$ and make it the first element ($l_{\alpha_1}$) in a new tuple say $w(n)$. Next, if the second element in $u(n)$ is already in $w(n)$, we move to the third element in $u(n)$, if not, then we add the second element of $u(n)$ to be the second element of $w(n)$. We continue this process and only add an element from $u(n)$ to $w(n)$ if it does not already appear in $w(n)$. The final resulting tuple $w(n)$ will be a $\overset{\circ}{\omega}{(n)}$-tuple and contain all the distinct lucky factors of $n$. We can write $w(n)=(l_{\alpha_{i_1}},l_{\alpha_{i_2}},\dots,l_{\alpha_{i_\kappa}})$ where $\kappa=\overset{\circ}{\omega}{(n)}$. Then $u(n)$ can also be rewritten as $u(n)=(l_{\alpha_{i_1}},l_{\alpha_{i_1+1}},\dots,l_{\alpha_{i_2}},l_{\alpha_{i_2+1}},\dots,l_{\alpha_{i_\kappa}},l_{\alpha_{i_\kappa+1}},\dots)$. Then for any tuple $t$ contained in $w(n)$ we can define an injective map from $t\mapsto r$, where $r$ is contained in $u(n)$. Let $t=(l_{\alpha_{i_j}},l_{\alpha_{i_{j+1}}},\dots,l_{\alpha_{i_{j+s}}})$ for some $s\geq 0$. Then we have $r=(l_{\alpha_{i_j}},l_{\alpha_{i_j+1}},\dots,l_{\alpha_{i_{j+1}}},l_{\alpha_{i_{j+1}+1}}\dots,l_{\alpha_{i_{j+s}}})$ is contained in $u(n)$ and the mapping $t \mapsto r$ is injective. Therefore, $w(n)$ contains fewer tuples than $u(n)$. It is easy to see that $w(n)$ contains one $\kappa$-tuples, two $(\kappa-1)$-tuples, three $(\kappa-2)$-tuples etc. Therefore we have that $w(n)$ contains $\sum\limits_{i=1}^{\overset{\circ }{\omega}(n)}{i}=\frac{(\overset{\circ }{\omega}(n))^2+\overset{\circ }{\omega}(n)}{2}$ tuples. This concludes the proof.

\end{proof}

\begin{defn}
    We call two integers $n$ and $m$ relatively lucky if there is no integer $d\in \mathbb{N}$ such that $d\overset{\circ }{|}n$ and $d\overset{\circ }{|}m$.
\end{defn} 

Recall that Euler's totient function $\phi(n)$ is the number of positive integers less than or equal to $n$ that are relatively prime to $n$. It is possible to define a similar function for lucky numbers.
\begin{defn}
    Define the lucky totient function $\overset{\circ }{\phi}(n)$ as the number of positive integers less than $n$ that are relatively lucky to $n$.
\end{defn}
Once again, the question arises, how does $\overset{\circ }{\phi}(n)$ relate to $\phi(n)$ and what is the average order of $\overset{\circ }{\phi}(n)$?

\subsection{On the Significance of Lucky Primes}
Lucky primes are prime numbers that are also lucky numbers. The first few lucky primes are $\{3,7,13,31,37,43,\dots\}$. It has been conjectured that there are infinitely many lucky primes \cite{r7}. We will show that every integer can be expressed (but not necessarily uniquely) as the \lq product\rq \, of combinations of powers of lucky primes (and 1) by using two binary operators, one being regular multiplication, the other will be the operator $*$. This result shows how the sequence of lucky primes is an even more special sequence of integers than the original prime numbers or lucky numbers on their own. Additionally, we find a generalization of the lucky prime conjecture which might provide a potential new approach to proving the infinitude of lucky primes.

\subsubsection*{Procedure to Factor an Integer in terms of Lucky Primes}
Start with any integer $n>1$. By the fundamental theorem of arithmetic, we can express $n$ as the product of powers of primes, say $n=p_{i_1}^{a_{i_1}}p_{i_2}^{a_{i_2}}\dots p_{i_{K_1}}^{a_{i_{K_1}}}$. Now each of the primes $p_{i_1}^{a_{i_1}},p_{i_2}^{a_{i_2}},\dots ,p_{i_{K_1}}^{a_{i_{K_1}}}$ can be written as the \lq product\rq \, of lucky numbers by using Theorem \ref{thmA}. For example, we can write $p_{i_1}=l_{\alpha_1}*(l_{\alpha_2}*(\cdots*l_{\alpha_{k_1}})\cdots)$. Then, we have $n=(l_{\alpha_1}*(l_{\alpha_2}*(\cdots*l_{\alpha_{k_1}})\cdots))^{a_{i_1}}p_{i_2}^{a_{i_2}}\dots p_{i_{K_1}}^{a_{i_{K_1}}}$. Each of the lucky numbers $l_{\alpha_{1}},l_{\alpha_{2}},\dots,l_{\alpha_{k_1}}$ can be written as the product of powers of primes. For example, we can write $l_{\alpha_{1}}=p_{j_1}^{a_{j_1}},p_{j_2}^{a_{j_2}},\dots,p_{j_{K_2}}^{a_{j_{K_2}}}$. Then, $n$ can be written as $n=((p_{j_1}^{a_{j_1}},p_{j_2}^{a_{j_2}},\dots,p_{j_{K_2}}^{a_{j_{K_2}}})*(l_{\alpha_2}*(\cdots*l_{\alpha_{k_1}})\cdots))^{a_{i_1}}p_{i_2}^{a_{i_2}}\dots p_{i_{K_1}}^{a_{i_{K_1}}}$. Note that the process will only terminate if we reach a lucky prime, since all prime numbers are written as themselves in the fundamental theorem of arithmetic and all lucky numbers are written as themselves in Theorem \ref{thmA}. Therefore, if we repeat this process for all the primes and lucky numbers when factoring $n$, we obtain a representation of $n$ in terms of only lucky primes (and possibly the number 1 since $1\in A_\infty$). 

This representation is not necessarily unique. This is because we can do the same factoring process, but instead start by first factoring $n$ in terms of lucky numbers. Furthermore, if we factor $n$ in any way in terms of its divisors (not necessarily prime divisors), and then start the algorithm by factoring each of those divisors in terms of lucky numbers, we may obtain a different representation. 
\begin{example}
    Let $n=77$. Then, $77=7(11)$, 7 is a lucky prime and therefore cannot be factored any further. However, we can write $11=3*(2*1)$. Thus, 
    $$
    77=7(3*(2*1)).
    $$
    However, also note that we have $$77=3*13$$ and since both 3 and 13 are lucky primes, this cannot be factored any further. So, the representation is not unique.
\end{example}

This representation may potentially be useful to use an Euclid-type approach to prove the infinitude of lucky primes, that is, assume that there are only finitely many lucky primes. Then use the fact that every integer can then be expressed as the product of only finitely many lucky primes and try to obtain a contradiction by finding a number that cannot be expressed as the product of just these finitely many lucky primes. However, this approach has many complications and has thus far not been successful.

One idea might be to restrict the representation of an integer in terms of lucky primes to only allow the factoring as explained above, where you factor the primes in terms of lucky numbers and lucky numbers in terms of primes. Therefore factoring $n$ or a lucky number in terms of just any arbitrary divisors is not allowed, it is only allowed to factor it in terms of powers of primes. For example, if $n=210=2*105$, then 105 is a lucky number. Now, writing $105=3(35)$ and then factoring $35$ in terms of lucky numbers is now allowed since 35 is not a prime. Therefore, with this restriction, there are at most two ways to factor an integer in terms of lucky primes.

With this restriction, we define a new concept called the fractional order of a prime/lucky number.
\begin{defn}
    Lucky numbers that are prime numbers have a prime order of 1. Prime numbers that are lucky numbers have a lucky order of 1, i.e. lucky primes have order 1. Define lucky numbers that are the product of powers of lucky primes to have a prime order of a half, and similarly define prime numbers that are the \lq product \rq \, of lucky primes to have lucky order of a half.
    In general, we define a prime (respectively lucky) number to have a lucky (respectively prime) order of $\frac{1}{n+1}$ if the smallest fractional order of its lucky (respectively prime) factors is $1/n$. sWe denote this by $\text{ord}(p)$ ($\text{ord}(l)$). I.e., if $p$ is a prime, $l$ is a lucky number and $p=l_{\alpha_1}*(l_{\alpha_2}*(\cdots*l_{\alpha_k})\cdots)$, $l=p_1^{a_1}p_2^{a_2}\cdots p_K^{a_K}$, then
    \begin{align*}
        \text{ord}(p)&= 1/\hk*{\max{\{\text{ord}^{-1}(l_{\alpha_1}),\text{ord}^{-1}(l_{\alpha_2}),\dots,\text{ord}^{-1}(l_{\alpha_k})\}}+1},\\
        \text{ord}(l)&= 1/\hk*{\max{\{\text{ord}^{-1}(p_{1}),\text{ord}^{-1}(p_2),\dots,\text{ord}^{-1}(p_K)\}}+1}.
    \end{align*}
\end{defn}
\begin{example}
    $41=3*7$, therefore 41 has a lucky order of 1/2.
\end{example}
This now gives a new conjecture more general than the lucky prime conjecture.
\begin{conj}
    Let $n$ be any fixed positive integer. There are infinitely many prime numbers (or lucky numbers) of order $1/n$.
\end{conj}
Note that for $n=1$, this is simply the lucky prime conjecture. However, solving this for, say $n=2$, might be easier than just proving the lucky prime conjecture.
\section{An Analogue of Bertrand's Postulate for Lucky Numbers} \label{s2}
We will prove the following theorem.
\begin{theorem}[Bertrand's Postulate for Lucky Numbers]\label{thmB}
    For any $x\geq 4$, there is always a lucky number in the interval $[x,2x]$.
\end{theorem}
We first need to prove the following lemma.
\begin{lemma} \label{lma1}
    For any $x\in [1,\infty), $
    $$x \prod_{i=1}^{{L(L(x))}}{\left(1-\frac{1}{l_i}\right)}\leq L(x)< x \prod_{i=1}^{{L(L(x))}}{\left(1-\frac{1}{l_i}\right)} +L(L(x)),$$
    where $l_{{L(L(x))}}$ is the largest lucky number so that $l_{{L(L(x))}}\leq L(x)$.
\end{lemma}

\begin{proof}
The lower bound follows directly by applying the formula from Proposition \ref{prop1}, and the fact that $y<\ceil{y}$ for all $y.$

For the upper bound, consider the sequence $R(i,x).$ Using the fact that $\ceil{y}\leq y+1$ for any $y\geq0$, it follows that $R(i+1,x)=\ceil*{R(i,x)-R(i,x)/l_{i+1}}\leq R(i,x)(1-1/l_{i+1})+1$ for all $i$.
Applying this at each step to Proposition \ref{prop1} gives 
\begin{align*}
     L(x)
     &=\ceil*{\dotsi \ceil*{\ceil*{\ceil*{\floor{x}\hk*{1-\frac{1}{l_1}}}\left(1-\frac{1}{l_2}\right)}\left(1-\frac{1}{l_3}\right)}\dotsi \left(1-\frac{1}{l_{L(L(x))}}\right)}\\
     & \leq \ceil*{\dotsi \ceil*{\ceil*{\hk*{\floor{x}\hk*{1-\frac{1}{l_1}}+1}\left(1-\frac{1}{l_2}\right)}\left(1-\frac{1}{l_3}\right)}\dotsi \left(1-\frac{1}{l_{L(L(x))}}\right)}\\
     & \leq \ceil*{\dotsi \ceil*{\hk*{\hk*{\floor{x}\hk*{1-\frac{1}{l_1}}+1}\left(1-\frac{1}{l_2}\right)+1}\left(1-\frac{1}{l_3}\right)}\dotsi \left(1-\frac{1}{l_{L(L(x))}}\right)}\\
     & \leq \ceil*{\dotsi \hk*{\hk*{\hk*{\floor{x}\hk*{1-\frac{1}{l_1}}+1}\left(1-\frac{1}{l_2}\right)+1}\left(1-\frac{1}{l_3}\right)+1}\dotsi \left(1-\frac{1}{l_{L(L(x))}}\right)}\\
     & \; \;\vdots\\
     & \leq \hk*{\dotsi \hk*{\hk*{\hk*{\floor{x}\hk*{1-\frac{1}{l_1}}+1}\left(1-\frac{1}{l_2}\right)+1}\left(1-\frac{1}{l_3}\right)+1}\dotsi \left(1-\frac{1}{l_{L(L(x))}}\right)}+1\\
     &= \hk*{\dotsi \hk*{\hk*{\floor{x}\hk*{1-\frac{1}{l_1}}\left(1-\frac{1}{l_2}\right)+\left(1-\frac{1}{l_2}\right)+1}\left(1-\frac{1}{l_3}\right)+1}\dotsi \left(1-\frac{1}{l_{L(L(x))}}\right)}+1\\
      & \; \;\vdots\\
      &=\floor{x}\prod_{i=1}^{{L(L(x))}}{\left(1-\frac{1}{l_i}\right)} +\sum_{i=1}^{{L(L(x))}}{\prod_{j=i}^{{L(L(x))}}{\left(1-\frac{1}{l_j}\right)}}.
\end{align*}
Note that $\prod\limits_{j=i}^{{L(L(x))}}{\left(1-\frac{1}{l_j}\right)}<1$ since each term in the product is less than one.
Hence 
\begin{equation}
    L(x) < x \prod_{i=1}^{{L(L(x))}}{\left(1-\frac{1}{l_i}\right)} +\sum_{i=1}^{{L(L(x))}}{1}.
\end{equation}
Finally we get the upper bound 
\begin{equation}
    L(x)<x \prod_{i=1}^{{L(L(x))}}{\left(1-\frac{1}{l_i}\right)} +L(L(x)).
\end{equation}
\end{proof}

\begin{proof}
[Proof (Theorem \ref{thmB})]
Assume that for some integer value of $x$, there is no lucky number between $x$ and $2x$. Thus $L(x)=L(2x)$.  For the sake of contradiction, we will show that for large enough $x$, $L(2x)-L(x)>0$ contradicting that there is no lucky number between $x$ and $2x$. 

Note that ${L(L(x))}={L(L(2x))}$ from the assumption $L(x)=L(2x)$. From the lower bound in Lemma \ref{lma1} for $L(2x)$
$$L(2x)\geq\frac{2x}{2} \prod_{i=2}^{{L(L(x))}}{\left(1-\frac{1}{l_i}\right)}=x\prod_{i=2}^{{L(L(x))}}{\left(1-\frac{1}{l_i}\right)}.$$
From the upper bound in Lemma \ref{lma1} for $L(x)$
$$L(x)<\frac{x}{2} \prod_{i=2}^{{L(L(x))}}{\left(1-\frac{1}{l_i}\right)} +L(L(x)).$$
Subtracting the previous two equations gives the following
\begin{align}
    L(2x)-L(x)&>\frac{2x}{2} \prod_{i=2}^{{L(L(x))}}{\left(1-\frac{1}{l_i}\right)}-\frac{x}{2} \prod_{i=2}^{{L(L(x))}}{\left(1-\frac{1}{l_i}\right)} -L(L(x)) \notag \\
    &>\frac{x}{2} \prod_{i=2}^{{L(L(x))}}{\left(1-\frac{1}{l_i}\right)} -L(L(x)).\label{eq32}
\end{align}
Rearranging the upper bound in Lemma \ref{lma1} 
$$ L(x)-L(L(x))<\frac{x}{2} \prod_{i=2}^{{L(L(x))}}{\left(1-\frac{1}{l_i}\right)}$$
and substituting this into equation (\ref{eq32}) we get 
\begin{align}
    L(2x)-L(x)&>\hk*{L(x)-L(L(x))}-L(L(x))\notag\\
    &=L(x)-2L(L(x)).\label{eq33}
\end{align}

Let $y=L(x)$. Note that for large enough $y$, by applying the formula from Proposition \ref{prop1}, it follows that $l_{{L(L(y))}} \geq 3$ for $y \geq 7$. Thus  
$$L(y) \leq \ceil*{\ceil*{y/2}2/3}\leq \ceil*{\hk*{y/2 +1}2/3}=\ceil*{y/3+2/3}\leq y/3+2/3+1\leq y/2 \text{ for } y\geq 10.$$ 

\begin{equation}
  \text{So for } y\geq10, L(y)\leq y/2=L(x)/2 
\end{equation}
Note that since $y=L(x)$, we have $L(y)=L(L(x))$. It then follows $L(L(x))=L(y)\leq L(x)/2$. Substitute this into equation (\ref{eq33}) to get 
\begin{equation}
L(2x)-L(x)>L(x)-2(L(x)/2)=0
\end{equation}
contradicting the assumption.
This happens whenever $y\geq \max\{7,10\}=10$, so when $L(x)\geq10$. This is when $x\geq 33$ since $L(33)=10$. By inspection of the lower cases, it follows that there is always a lucky number between $x$ and $2x$ for all $x\geq 4$.
\end{proof}

\section{On the Order of Magnitude of the Difference Between Consecutive Lucky Numbers} \label{s3}
In this section, we prove the following theorem.
\begin{theorem}\label{thm4.2}
    For sufficiently large $n$, the following holds 
    $$l_{n+1}-l_n=O\hk*{\sqrt{{l_n}\log{l_n}}},$$
    or equivalently, using the lucky number theorem 
    $$l_{n+1}-l_n=O\hk*{\sqrt{n}\log{n}}.$$
\end{theorem}
Start with any $n \in \mathbb{N}$. The requirement is to find an upper bound for $l_{n+1}-l_n$. The proof uses the formula in Proposition \ref{prop2} (\ref{eq3.11}) (this will be referred to as formula (\ref{eq3.11}) for the remainder of the proof). Similarly to Proposition \ref{prop2}, we define $\beta=L(n)$ so that $l_{\beta} $ is the largest lucky number less than or equal to $n$. With this in mind, we define a new sequence. 

    Define the finite sequence $\hk*{n_i}_{0}^{\beta}$, where $n_0=n$ and 
    $$n_{i+1}=\floor*{\frac{n_i l_{\beta-i}-1}{l_{\beta-i}-1}}.$$
    Where $l_\beta$ is the largest lucky number less than or equal to $n$.

Note that this is an increasing sequence, so $n_{i+1}\geq n_i$ and iterating this sequence gives us formula (\ref{eq3.11}). Therefore, it follows that $n_\beta=l_n$.
Consider formula (\ref{eq3.11}) applied to $n+1$. This gives the sequence (as in the definition) $\hk*{(n+1)_i}_{0}^{\beta'}$, where $(n+1)_0=n+1$ and $\beta'=L(n+1)$, thus $l_{\beta'}$ is the largest lucky number less than or equal to $n+1$. The objective is to find an upper bound for $l_{n+1}$ in terms of $l_n$. This will be done by finding an upper bound for each term $(n+1)_i$ in terms of some term in the sequence $\hk*{n_i}_{0}^{\beta}$. Note that the biggest difference between $l_{n+1}$ and $l_n$ could be when $l_{\beta'}=n+1$ (since this would imply that $\beta'=\beta+1$, hence there are more terms in $\hk*{(n+1)_i}_{0}^{\beta'}$ than there are in $\hk*{n_i}_{0}^{\beta}$ and the sequences are increasing). Therefore, the assumption will be made that $l_{\beta'}=n+1$ (since the proof method will yield a smaller upper bound if not). 

Using the definition and the fact that $n+1=(n+1)_0$ and $l_{\beta'}=n+1$ 
$$
(n+1)_1=\floor*{\frac{(n+1)(n+1)-1}{(n+1)-1}}=\floor*{\frac{n^2+2n}{n}}=n+2.
$$
Substituting this, it follows that
$$
(n+1)_2=\floor*{\frac{(n+2)l_\beta-1}{l_\beta-1}}=\floor*{\frac{nl_\beta-1}{l_\beta-1}+\frac{2l_\beta}{l_\beta-1}}.
$$
Note that the floor function satisfies the following property: for any $x,y\in \mathbb{R},$
\begin{equation} \label{eq37}
\floor{x}+\floor{y}\leq \floor{x+y}\leq \floor{x}+\floor{y}+1.
\end{equation}
This property inspires the following definition of a sequence of errors. Define the sequence $(e_i)$ as follows, $e_1=2$, so that $(n+1)_1 \leq n_0+e_1=n+2$.

For $i\geq 2$, this pattern continues in such a way that $e_{i}$ depends on $e_{i-1}$. Applying the definition of $(n_i),\, ((n+1)_i)$, 

\begin{align}
    (n+1)_2 &\leq \floor*{\frac{(n_0+e_1)l_\beta-1}{l_\beta-1}}=\floor*{\frac{n_0l_\beta-1}{l_\beta-1}+\frac{2l_\beta}{l_\beta-1}} \leq \floor*{\frac{n_0l_\beta-1}{l_\beta-1}}+e_2=n_1+e_2, \notag \\
    (n+1)_3 &\leq \floor*{\frac{(n_1+e_2)l_{\beta-1}-1}{l_{\beta-1}-1}}=\floor*{\frac{n_1l_{\beta-1}-1}{l_{\beta-1}-1}+\frac{e_2l_{\beta-1}}{l_{\beta-1}-1}} \leq \floor*{\frac{n_1l_{\beta-1}-1}{l_{\beta-1}-1}}+e_3=n_2+e_3, \notag \\
    \; \; \; \; \; \vdots &\notag \\
    (n+1)_i &\leq \floor*{\frac{(n_{i-2}+e_{i-1})l_{\beta-i+2}-1}{l_{\beta-i+2}-1}}=\floor*{\frac{n_{i-2}l_{\beta-i+2}-1}{l_{\beta-i+2}-1}+\frac{e_{i-1}l_{\beta-i+2}}{l_{\beta-i+2}-1}} \notag \\
    &\leq \floor*{\frac{n_{i-2}l_{\beta-i+2}-1}{l_{\beta-i+2}-1}}+e_{i}=n_{i-1}+e_{i}. \label{eq38}
\end{align}
Therefore, it follows that 
\begin{equation} \label{eq39}
e_i\geq \floor*{\frac{(n_{i-2}+e_{i-1})l_{\beta-i+2}-1}{l_{\beta-i+2}-1}}-n_{i-1}.
\end{equation}
Note that $e_i$ is not necessarily well defined, instead the values of $(e_i)$ are our own choice and can be anything that gives an upper bound as shown above. Therefore, any error as in equation (\ref{eq38},\ref{eq39}) will suffice. The important part is that $e_i$ depends on $e_{i-1}$ to ensure that the upper bound $(n+1)_i\leq n_{i-1}+e_{i}$ still holds.  As a side note, if $n+1\neq l_{\beta'}$, then the sequence can be defined in the same manner except where $e_i=1$ and the inequality $(n+1)_i\leq n_{i}+e_{i}$ should hold. Since $e_{i+1}$ depends on $e_i$, we will not consider this case since it gives a smaller error term.
\begin{prop}
    Note that by equation (\ref{eq37}), to obtain the minimum value for the error, we have
\begin{equation} \label{eq40}
    e_{i}\in \brac*{\floor*{\frac{e_{i-1}l_{\beta-i+2}}{l_{\beta-i+2}-1}},\floor*{\frac{e_{i-1}l_{\beta-i+2}}{l_{\beta-i+2}-1}}+1}.
\end{equation}
\end{prop}

Consider the following calculations. Let 
\begin{equation} \label{eq41}
n_{i-2}=q_{i-2}(l_{\beta-i+2}-1)+r_{i-2}, \text{ where } 0\leq r_{i-2}\leq l_{\beta-i+2}-2, i\geq2.
\end{equation}
Substituting this, we have 
\begin{align}
    (n+1)_{i} &\leq \floor*{\frac{n_{i-2}l_{\beta-i+2}-1}{l_{\beta-i+2}-1}+\frac{e_{i-1}l_{\beta-i+2}}{l_{\beta-i+2}-1}} \notag \\
    &= \floor*{\frac{(q_{i-2}(l_{\beta-i+2}-1)+r_{i-2})l_{\beta-i+2}-1}{l_{\beta-i+2}-1}+\frac{e_{i-1}l_{\beta-i+2}}{l_{\beta-i+2}-1}}\notag\\
    &=q_{i-2}l_{\beta-i+2}+\floor*{\frac{r_{i-2}l_{\beta-i+2}-1}{l_{\beta-i+2}-1}+\frac{e_{i-1}(l_{\beta-i+2}-1)+e_{i-1}}{l_{\beta-i+2}-1}}\notag\\
    &=q_{i-2}l_{\beta-i+2}+e_{i-1}+\floor*{\frac{r_{i-2}(l_{\beta-i+2}-1)+r_{i-2}-1}{l_{\beta-i+2}-1}+\frac{e_{i-1}}{l_{\beta-i+2}-1}}\notag\\
    &=q_{i-2}l_{\beta-i+2}+r_{i-2}+e_{i-1}+\floor*{\frac{r_{i-2}-1+e_{i-1}}{l_{\beta-i+2}-1}}.\label{eq42}
\end{align}
Using the definition of $(n_i)$ 
\begin{align}
    n_{i-1} &=\floor*{\frac{n_{i-2}l_{\beta-i+2}-1}{l_{\beta-i+2}-1}} \notag \\
    &=  \floor*{\frac{(q_{i-2}(l_{\beta-i+2}-1)+r_{i-2})l_{\beta-i+2}-1}{l_{\beta-i+2}-1}} \notag\\
    &=q_{i-2}l_{\beta-i+2}+\floor*{\frac{r_{i-2}(l_{\beta-i+2}-1)+r_{i-2}-1}{l_{\beta-i+2}-1}}\notag \\
    &=q_{i-2}l_{\beta-i+2}+r_{i-2}+\floor*{\frac{r_{i-2}-1}{l_{\beta-i+2}-1}}.\label{eq43}
\end{align}
Note the following 
\begin{equation} \label{eq44}
    \floor*{\frac{r_{i-2}-1}{l_{\beta-i+2}-1}}=
    \begin{cases}
        -1,& \text{if } r_{i-2}=0;\\
        0,& \text{otherwise.}
    \end{cases}
\end{equation}
Using equation (\ref{eq43}) and (\ref{eq44}) gives the following result: 
\begin{align}
    &\floor*{\frac{n_{i-2}l_{\beta-i+2}-1}{l_{\beta-i+2}-1}}+\floor*{\frac{e_{i-1}l_{\beta-i+2}}{l_{\beta-i+2}-1}}=q_{i-2}l_{\beta-i+2}+r_{i-2}+\floor*{\frac{r_{i-2}-1}{l_{\beta-i+2}-1}}+e_{i-1}+\floor*{\frac{e_{i-1}}{l_{\beta-i+2}-1}}\notag \\
    &=\begin{cases}
        q_{i-2}l_{\beta-i+2}+r_{i-2}+e_{i-1}-1+\floor*{\frac{e_{i-1}}{l_{\beta-i+2}-1}},& \text{if } r_{i-2}=0;\\
        q_{i-2}l_{\beta-i+2}+r_{i-2}+e_{i-1}+\floor*{\frac{e_{i-1}}{l_{\beta-i+2}-1}},& \text{otherwise.}
    \end{cases}\label{eq45}
\end{align} 
With this in mind, the next objective is to determine when is
\begin{equation}\label{eq4.10}
   \floor*{\frac{n_{i-2}l_{\beta-i+2}-1}{l_{\beta-i+2}-1}+\frac{e_{i-1}l_{\beta-i+2}}{l_{\beta-i+2}-1}}
    =\floor*{\frac{n_{i-2}l_{\beta-i+2}-1}{l_{\beta-i+2}-1}}+\floor*{\frac{e_{i-1}l_{\beta-i+2}}{l_{\beta-i+2}-1}}? 
\end{equation}
By equations (\ref{eq42}) and (\ref{eq45}), this is when, 
$$
\floor*{\frac{r_{i-2}-1+e_{i-1}}{l_{\beta-i+2}-1}}=\floor*{\frac{e_{i-1}}{l_{\beta-i+2}-1}} \text{, if } r_{i-2}\neq0
$$
and
$$
\floor*{\frac{e_{i-1}-1}{l_{\beta-i+2}-1}}=\floor*{\frac{e_{i-1}}{l_{\beta-i+2}-1}}-1 \text{, if } r_{i-2}=0
$$
To simplify the proof, for the case where $r_{i-2}=0$, use equation (\ref{eq37}) to get the upper bound 
\begin{align} \label{eq46}
    (n+1)_{i} &\leq \floor*{\frac{n_{i-2}l_{\beta-i+2}-1}{l_{\beta-i+2}-1}+\frac{e_{i-1}l_{\beta-i+2}}{l_{\beta-i+2}-1}}
    \leq \floor*{\frac{n_{i-2}l_{\beta-i+2}-1}{l_{\beta-i+2}-1}}+\floor*{\frac{e_{i-1}l_{\beta-i+2}}{l_{\beta-i+2}-1}}+1\\
    &=n_{i-1}+e_i.\notag
\end{align}
Therefore, it follows that
\begin{equation} \label{eq47}
    e_i=\floor*{\frac{e_{i-1}l_{\beta-i+2}}{l_{\beta-i+2}-1}}+1.
\end{equation}
Consider the case where $r_{i-2}\neq0$. Note that by equation (\ref{eq37}), it follows that either 
$$
\floor*{\frac{r_{i-2}-1+e_{i-1}}{l_{\beta-i+2}-1}}=\floor*{\frac{e_{i-1}}{l_{\beta-i+2}-1}},
$$
or
\begin{equation}\label{eq48}
    \floor*{\frac{r_{i-2}-1+e_{i-1}}{l_{\beta-i+2}-1}}=\floor*{\frac{e_{i-1}}{l_{\beta-i+2}-1}}+1.
\end{equation}
Therefore, based on this, it is relevant to rather ask when is 
$$
\floor*{\frac{r_{i-2}-1+e_{i-1}}{l_{\beta-i+2}-1}}=\floor*{\frac{e_{i-1}}{l_{\beta-i+2}-1}}+1?
$$
This can be split into two simple cases. \\
Case A:
if $e_{i-1} < l_{\beta-i+2}-1$. Then clearly
$$
\floor*{\frac{e_{i-1}}{l_{\beta-i+2}-1}}+1=1.
$$
Then equation (\ref{eq48}) can only hold when 
$$
r_{i-2}-1+e_{i-1}\geq l_{\beta-i+2}-1
$$
so that
$$
r_{i-2}\geq l_{\beta-i+2}-e_{i-1}.
$$
Combining this with the quotient remainder theorem gives 
\begin{equation} \label{eq49}
l_{\beta-i+2}-e_{i-1}\leq r_{i-2}\leq l_{\beta-i+2}-2.
\end{equation}
Case B: 
if $e_{i-1} \geq l_{\beta-i+2}-1$. Then simply use the fact that equation (\ref{eq48}) can hold only for the remainders 
\begin{equation} \label{eq50}
0\leq r_{i-2}\leq l_{\beta-i+2}-2<e_{i-1}.
\end{equation}
Recall (\ref{eq41}). Note that in some cases, it is possible for some quotients to satisfy the following \\ $q_{i-2}=q_{i-1}=\cdots=q_{i+j-2}$ for some integers $i,j$.
Applying equation (\ref{eq41}) to the index values, $i-2,i-1,\dots,i+j-2$, we get 
\begin{align*}
    &n_{i-2}=q_{i-2}(l_{\beta-i+2}-1)+r_{i-2}, \text{ where } 0\leq r_{i-2}\leq l_{\beta-i+2}-2, i\geq2.\\
    &n_{i-1}=q_{i-1}(l_{\beta-i+1}-1)+r_{i-1}, \text{ where } 0\leq r_{i-1}\leq l_{\beta-i+1}-2.\\
    &\;\;\;\vdots\\
    &n_{i+j-2}=q_{i+j-2}(l_{\beta-i-j+2}-1)+r_{i+j-2}, \text{ where } 0\leq r_{i+j-2}\leq l_{\beta-i-j+2}-2.
\end{align*}
Using the fact that $q_{i-2}=q_{i-1}=\cdots=q_{i+j-2}$, we have 
\begin{align*}
    &n_{i-2}=q_{i-2}(l_{\beta-i+2}-1)+r_{i-2}, \text{ where } 0\leq r_{i-2}\leq l_{\beta-i+2}-2, i\geq2.\\
    &n_{i-1}=q_{i-2}(l_{\beta-i+1}-1)+r_{i-1}, \text{ where } 0\leq r_{i-1}\leq l_{\beta-i+1}-2.\\
    &\;\;\;\vdots\\
    &n_{i+j-2}=q_{i-2}(l_{\beta-i-j+2}-1)+r_{i+j-2}, \text{ where } 0\leq r_{i+j-2}\leq l_{\beta-i-j+2}-2.
\end{align*}
Keeping this in mind, subtract the first two equations to get 
$$
n_{i-1}-n_{i-2}=q_{i-2}\hk*{l_{\beta-i+1}-l_{\beta-i+2}}+r_{i-1}-r_{i-2}.
$$
Note that $(n_i)$ is a strictly increasing sequence, hence $n_{i-1}-n_{i-2}\geq 1,$ and $l_{\beta-i+1}-l_{\beta-i+2}\leq-2$, it follows that 
\begin{equation} 
    1+2q_{i-2}\leq r_{i-1}-r_{i-2} \implies r_{i-2}+1+2q_{i-2}\leq r_{i-1}.
\end{equation} 
Or, more simply, for any $i-1\leq \tau \leq i+j-2,$ 
\begin{equation} 
r_{\tau-1}+1+2q_{i-2}\leq r_{\tau}.
\end{equation} 
Iterating this recurrence relation: 
\begin{align*}
    r_{i+j-2} &\geq r_{i+j-3}+1+2q_{i-2}\\
    &\geq \hk*{r_{i+j-4}+1+2q_{i-2}}+1+2q_{i-2}=r_{i+j-4}+2\hk*{1+2q_{i-2}}\\
    &\geq r_{i+j-5}+3\hk*{1+2q_{i-2}}\\
    &\;\,\vdots\\
    &\geq r_{i-2}+j\hk*{1+2q_{i-2}}
\end{align*}
Therefore, 
\begin{equation} \label{eq53}
     r_{i+j-2}\geq r_{i-2}+j\hk*{1+2q_{i-2}}.
\end{equation}

Define the following finite sequence of quotients, $(k_i)_{i=1}^{N}$, where 
\begin{equation} \label{eq54}
    N=\ceil*{\frac{\sqrt{l_n}}{\sqrt{\log (l_n)}}}.
\end{equation}
 Let $k_1=q_0=q_1=\cdots=q_{m_1}$, $k_2=q_{m_1+1}=q_{m_1+2}=\cdots=q_{m_2}$, $k_3=q_{m_2+1}=q_{m_2+2}=\cdots=q_{m_3}$, $\cdots$ and so on. Note that this is a subsequence of $(q_i)$. Therefore, every quotient $k_i$ can have multiple corresponding lucky numbers and remainders (this would be the lucky numbers and remainders that correspond to each quotient $q_i$). We also have that $(k_i)$ is a strictly increasing sequence of integers (by construction), and $k_i\geq i$ for all $i$.

Using this new sequence, equation (\ref{eq53}) can be rewritten as 
\begin{equation} \label{eq55}
     r_{i+j-2}\geq r_{i-2}+j\hk*{1+2k}, \text{ for some } k\in (k_i), \text{where }
\end{equation}
$q_{i-2}=q_{i-1}=\cdots=q_{i+j-2}=k$ and note that $j$ comes from the subscripts of the remainders in the equation where $(i+j-2)-(i-2)=j$. Also note that the remainders corresponding to each $k$ are strictly increasing.

The next step is to find an upper bound for each $k\in (k_i)$ on how many corresponding steps (or remainders or lucky numbers) of formula (\ref{eq3.11}) does equation (\ref{eq48}) hold. To do this, it will be assumed that for each $k$, there is a corresponding remainder of 0. It will also be assumed that for each $k_j$, there is a corresponding remainder $r_{i_j}$ such that equations (\ref{eq49}) (case A) and (\ref{eq50}) (case B) hold, $l_{\beta-i_j}-e_{i_j+1}\leq r_{i_j}\leq l_{\beta-i_j}-2$ and $0\leq r_{i_j}\leq l_{\beta-i_j}-2<e_{i_j+1}.$ Note that since the remainders are increasing for each $k_j$, in order to find an upper bound, assume $r_{i_j}$ to be as small as possible. Hence, by these two cases and equation (\ref{eq55}), it follows that an upper bound for each $k_j$ on how many times equation (\ref{eq48}) holds is 
\begin{equation} \label{eq56}
    2+\floor*{\frac{e_{i_j+1}-1}{1+2k_j}}.
\end{equation}
The first two cases are for the remainder of 0 and for $r_{i_j}$ and the term $\floor*{\frac{e_{i_j+1}-1}{1+2k_j}}$ comes from the fact that cases A and B give an interval of less than or equal to $e_{i_j+1}$ for the remainders for which equation (\ref{eq48}) holds and then by applying equation (\ref{eq55}), the result follows.

Also note that in all other cases, and by equation (\ref{eq4.10}) the following can be applied 
\begin{align}
(n+1)_i&\leq \floor*{\frac{n_{i-2}l_{\beta-i+2}-1}{l_{\beta-i+2}-1}+\frac{e_{i-1}l_{\beta-i+2}}{l_{\beta-i+2}-1}} \leq \floor*{\frac{n_{i-2}l_{\beta-i+2}-1}{l_{\beta-i+2}-1}}+{\frac{e_{i-1}l_{\beta-i+2}}{l_{\beta-i+2}-1}}=n_{i-1}+e_i.\label{eq57}
\end{align} 

Define the following sequence $(s_i)_{i=1}^{N+1}$ where $s_1=2=e_1$ and then for $i\geq2$, 
\begin{equation}\label{eq62}
    s_i=2+s_{i-1}+\frac{s_{i-1}+1}{2\hk{k_i-1}+1}.
\end{equation}

\begin{lemma} \label{lma5.2}
For any $j$ where $1\leq j \leq N$, we have
    \begin{equation*}
    (n+1)_{m_j}\leq n_{m_j-1}+\hk*{s_{j+1}}\prod_{i=0}^{m_j-2}{\hk*{\frac{l_{\beta-i}}{l_{\beta-i}-1}}},
    \end{equation*}
\end{lemma}
\begin{proof}
Create a subsequence $(e_{i_j})$ where $e_{i_j}$ is the first term in $(e_i)$ in which equation (\ref{eq48}) holds for the quotient $k_j$.
Start with $k_1=1$ and assuming the first remainder is 0: 
\begin{align*}
    (n+1)_2&\leq n_1+e_1+\floor*{\frac{e_1}{l_{\beta}-1}}+1, \text{ (by (\ref{eq46}))}\\
    &\leq n_1+e_1+{\frac{e_1-1}{l_{\beta}-1}}+1\\
    &=n_1+e_1\hk*{\frac{l_{\beta}}{l_{\beta}-1}}+1.
\end{align*}
So in this case, we can choose $e_2=e_1\hk*{\frac{l_{\beta}}{l_{\beta}-1}}+1$.
If $e_2\neq e_{i_1}$, then use equation (\ref{eq57}):
\begin{align*}
    (n+1)_3&\leq n_2+e_2\hk*{\frac{l_{\beta-1}}{l_{\beta-1}-1}}\\
    &=n_2+\hk*{e_1\hk*{\frac{l_{\beta}}{l_{\beta}-1}}+1}\hk*{\frac{l_{\beta-1}}{l_{\beta-1}-1}}\\
    &=n_2+e_1\hk*{\frac{l_{\beta}}{l_{\beta}-1}}\hk*{\frac{l_{\beta-1}}{l_{\beta-1}-1}}+\hk*{\frac{l_{\beta-1}}{l_{\beta-1}-1}}.
\end{align*}
Hence, we choose $e_3=e_1\hk*{\frac{l_{\beta}}{l_{\beta}-1}}\hk*{\frac{l_{\beta-1}}{l_{\beta-1}-1}}+\hk*{\frac{l_{\beta-1}}{l_{\beta-1}-1}}$. If $e_3\neq e_{i_1}$, then once again use equation (\ref{eq57}) and repeat this process again. 
\begin{align*}
    (n+1)_4&\leq n_3+e_3\hk*{\frac{l_{\beta-2}}{l_{\beta-2}-1}}\\
    &=n_3+\hk*{e_1\hk*{\frac{l_{\beta}}{l_{\beta}-1}}\hk*{\frac{l_{\beta-1}}{l_{\beta-1}-1}}+\hk*{\frac{l_{\beta-1}}{l_{\beta-1}-1}}}\hk*{\frac{l_{\beta-2}}{l_{\beta-2}-1}}\\
    &=n_3+e_1\hk*{\frac{l_{\beta}}{l_{\beta}-1}}\hk*{\frac{l_{\beta-1}}{l_{\beta-1}-1}}\hk*{\frac{l_{\beta-2}}{l_{\beta-2}-1}}+\hk*{\frac{l_{\beta-1}}{l_{\beta-1}-1}}\hk*{\frac{l_{\beta-2}}{l_{\beta-2}-1}}.
\end{align*}
Repeat this process for all $i< i_1$ to get 
\begin{align}
    (n+1)_{i_1}&\leq n_{i_1-1}+e_1\prod_{i=0}^{i_1-2}{\hk*{\frac{l_{\beta-i}}{l_{\beta-i}-1}}}+\prod_{i=1}^{i_1-2}{\hk*{\frac{l_{\beta-i}}{l_{\beta-i}-1}}}\notag \\
    &\leq n_{i_1-1}+(e_1+1)\prod_{i=0}^{i_1-2}{\hk*{\frac{l_{\beta-i}}{l_{\beta-i}-1}}}\notag\\
    &=n_{i_1-1}+3\prod_{i=0}^{i_1-2}{\hk*{\frac{l_{\beta-i}}{l_{\beta-i}-1}}}.    
    \label{eq58}
\end{align}
We can therefore set $e_{i_1}=(e_1+1)\prod\limits_{i=0}^{i_1-2}{\hk*{\frac{l_{\beta-i}}{l_{\beta-i}-1}}}$.

Next, we have to use equation (\ref{eq46}) to obtain an upper bound 
\begin{align*}
    (n+1)_{i_1+1}&\leq n_{i_1}+\floor*{\frac{e_{i_1}l_{\beta-i_1+1}}{l_{\beta-i_1+1}-1}}+1\\
    &\leq n_{i_1}+{\frac{e_{i_1}l_{\beta-i_1+1}}{l_{\beta-i_1+1}-1}}+1\\
    &=n_{i_1}+(e_1+1)\prod_{i=0}^{i_1-1}{\hk*{\frac{l_{\beta-i}}{l_{\beta-i}-1}}}+1.
\end{align*}
Therefore, we can choose the error $e_{i_1+1}=(e_1+1)\prod\limits_{i=0}^{i_1-1}{\hk*{\frac{l_{\beta-i}}{l_{\beta-i}-1}}}+1$. Note that this is the first of less than or equal to $1+\floor*{\frac{e_{i_1}-1}{1+2k_1}}$ cases by expression (\ref{eq56}). Repeating this process and using equation (\ref{eq46}), it follows that
\begin{align*}
    (n+1)_{i_1+2}&\leq n_{i_1+1}+(e_1+1)\prod_{i=0}^{i_1}{\hk*{\frac{l_{\beta-i}}{l_{\beta-i}-1}}}+\hk*{\frac{l_{\beta-i_1}}{l_{\beta-i_1}-1}}+1\\
    &\leq n_{i_1+1}+(e_1+1)\prod_{i=0}^{i_1}{\hk*{\frac{l_{\beta-i}}{l_{\beta-i}-1}}}+2\hk*{\frac{l_{\beta-i_1}}{l_{\beta-i_1}-1}}.
\end{align*}
Similarly, we have 
\begin{align}
    &(n+1)_{i_1+3} \leq n_{i_1+2}+(e_1+1)\prod_{i=0}^{i_1+1}{\hk*{\frac{l_{\beta-i}}{l_{\beta-i}-1}}}+3\hk*{\frac{l_{\beta-i_1}}{l_{\beta-i_1}-1}}\hk*{\frac{l_{\beta-i_1-1}}{l_{\beta-i_1-1}-1}}\notag \\
    &\;\;\,\ \;\;\vdots\notag\\
    &(n+1)_{m_1} \leq n_{m_1-1}+(e_1+1)\prod_{i=0}^{m_1-2}{\hk*{\frac{l_{\beta-i}}{l_{\beta-i}-1}}}+(m_1-i_1)\prod_{i=i_1}^{m_1-2}{\hk*{\frac{l_{\beta-i}}{l_{\beta-i}-1}}}.\label{eq59}
\end{align}
In equation (\ref{eq59}), $m_1$ is the largest index corresponding to the quotient $k_1$ and similarly $m_i$ is the largest index corresponding to the quotient $k_i$. From expression (\ref{eq56}), it follows that 
$$m_1\leq i_1+1+\floor*{\frac{e_{i_1}-1}{1+2k_1}}\leq i_1+1+\frac{e_{i_1}-1}{1+2k_1}.$$ 
Substituting the equation $e_{i_1}=(e_1+1)\prod\limits_{i=0}^{i_1-2}{\hk*{\frac{l_{\beta-i}}{l_{\beta-i}-1}}}$ in the preceding equation, obtain 
\begin{align}
    m_1&\leq i_1+1+\hk*{(e_1+1)\prod_{i=0}^{i_1-2}{\hk*{\frac{l_{\beta-i}}{l_{\beta-i}-1}}}}\hk*{\frac{1}{1+2k_1}}-\frac{1}{1+2k_1}\notag\\
    &\leq i_1+1+\hk*{\frac{e_1+1}{1+2k_1}}\prod_{i=0}^{i_1-1}{\hk*{\frac{l_{\beta-i}}{l_{\beta-i}-1}}}.\label{eq60}
\end{align}
Substituting this into equation (\ref{eq59}) yields 
\begin{align}
    (n+1)_{m_1} &\leq n_{m_1-1}+(e_1+1)\prod_{i=0}^{m_1-2}{\hk*{\frac{l_{\beta-i}}{l_{\beta-i}-1}}}+\hk*{1+\hk*{\frac{e_1+1}{1+2k_1}}\prod_{i=0}^{i_1-1}{\hk*{\frac{l_{\beta-i}}{l_{\beta-i}-1}}}}\prod_{i=i_1}^{m_1-2}{\hk*{\frac{l_{\beta-i}}{l_{\beta-i}-1}}}\notag\\
    &= n_{m_1-1}+(e_1+1)\prod_{i=0}^{m_1-2}{\hk*{\frac{l_{\beta-i}}{l_{\beta-i}-1}}}+\hk*{1+\hk*{\frac{e_1+1}{1+2k_1}}}\prod_{i=0}^{m_1-2}{\hk*{\frac{l_{\beta-i}}{l_{\beta-i}-1}}}\notag \\
    &=n_{m_1-1}+\hk*{e_1+2+\hk*{\frac{e_1+1}{1+2k_1}}}\prod_{i=0}^{m_1-2}{\hk*{\frac{l_{\beta-i}}{l_{\beta-i}-1}}}.\label{eq61}
\end{align}
Using induction on $j$ and applying the same process as before (this can be seen in equation (\ref{eq61})), the result follows.
\end{proof}
 
Substituting $j=N$ in Lemma \ref{lma5.2}, we have 
\begin{equation}
    (n+1)_{m_N}\leq n_{m_N-1}+\hk*{s_{N+1}}\prod_{i=0}^{m_N-2}{\hk*{\frac{l_{\beta-i}}{l_{\beta-i}-1}}}.\label{eq63}
\end{equation}
Note, this is the inequality corresponding to the quotient $k_N\geq N=\ceil*{\frac{\sqrt{l_n}}{\sqrt{\log (l_n)}}}$. By the way $m_N$ is defined, this is also the last step corresponding to this quotient.

\begin{lemma}
For any $n \in \mathbb{N}$,
    $$s_{n+1}=O(n).$$
\end{lemma}
We will use this to find an upper bound for $s_{N+1}$. 
\begin{proof}
From equation (\ref{eq62}), it follows that 
\begin{align}
    s_i&=2+s_{i-1}+\frac{s_{i-1}+1}{2\hk{k_i-1}+1}\leq 2+s_{i-1}+\frac{s_{i-1}+1}{2\hk{i-1}+1}\notag\\
    &\leq s_{i-1}+2+\frac{s_{i-1}+2}{2\hk{i-1}+1}\notag\\
    &=\hk*{s_{i-1}+2}\hk*{1+\frac{1}{2\hk{i-1}+1}}.\label{eq64}
\end{align}
Starting at $s_{N+1}$ and applying equation (\ref{eq64}) repeatedly: 
\begin{align*}
    S_{N+1}&\leq \hk*{s_{N}+2}\hk*{1+\frac{1}{2N+1}}\\
    &= 2\hk*{1+\frac{1}{2N+1}}+s_N\hk*{1+\frac{1}{2N+1}}\\
    & \leq 2\hk*{1+\frac{1}{2N+1}}+\hk*{s_{N-1}+2}\hk*{1+\frac{1}{2\hk{N-1}+1}}\hk*{1+\frac{1}{2N+1}} \text{, (by (\ref{eq64}))}\\
    &= 2\hk*{\frac{2N+2}{2N+1}}+2\hk*{\frac{2\hk{N-1}+2}{2\hk{N-1}+2}}\hk*{\frac{2N+1}{2N+1}}+s_{N-1}\hk*{\frac{2\hk{N-1}+2}{2\hk{N-1}+2}}\hk*{\frac{2N+1}{2N+1}}.
\end{align*}
Repeat this process to obtain 
\begin{align*}
   s_{N+1} &\leq 2\hk*{\hk*{\frac{2N+2}{2N+1}}+\hk*{\frac{2\hk{N-1}+2}{2\hk{N-1}+2}}\hk*{\frac{2N+1}{2N+1}}+\cdots+\prod_{j=1}^{N}{\hk*{\frac{2j+2}{2j+1}}}}\\
   &+s_1\prod_{j=1}^{N}{\hk*{\frac{2j+2}{2j+1}}}\\
   &=2\sum_{c=1}^{N}{\prod_{j=c}^{N}{\hk*{\frac{2j+2}{2j+1}}}}+2\prod_{j=1}^{N}{\hk*{\frac{2j+2}{2j+1}}},\, s_1=2.
\end{align*}
Since $2\prod\limits_{j=1}^{N}{\hk*{\frac{2j+2}{2j+1}}}$ is just a single term of the preceding sum, it follows that 
\begin{equation}\label{eq65}
    s_{N+1} \leq 4\sum_{c=1}^{N}{\prod_{j=c}^{N}{\hk*{\frac{2j+2}{2j+1}}}}.
\end{equation}
To find an upper bound for this, consider the product
\begin{equation}\label{eq66}
    P(c)=\prod_{j=c}^{N}{\hk*{\frac{2j+2}{2j+1}}}=\prod_{j=c}^{N}{\hk*{1+\frac{1}{2j+1}}}, \, c\geq1.
\end{equation}
Taking the natural log on both sides and using the inequality $\log{(1+x)}\leq x$ for $x\geq 0$,
\begin{align}
    \log{P(c)}&=\sum_{j=c}^{N}{\log{\hk*{1+\frac{1}{2j+1}}}} \leq \sum_{j=c}^{N}{\frac{1}{2j+1}} \notag\\
    &=\sum_{j=1}^{N}{\frac{1}{2j+1}}-\sum_{j=1}^{c-1}{\frac{1}{2j+1}}.\label{eq67}
\end{align}
Note that $\frac{1}{2j+1}<\frac{1}{2j}$, hence
$$
\sum_{j=1}^{N}{\frac{1}{2j}}>\sum_{j=1}^{N}{\frac{1}{2j+1}}.
$$
Adding $\sum\limits_{j=1}^{N}{\frac{1}{2j+1}}$ to both sides of the previous inequality, we have
$$\sum_{j=2}^{2N+1}{\frac{1}{j}}>2\sum_{j=1}^{N}{\frac{1}{2j+1}}$$
i.e.
$$\frac{1}{2}\sum_{j=2}^{2N+1}{\frac{1}{j}}>\sum_{j=1}^{N}{\frac{1}{2j+1}}.$$
It is possible to use the property $\sum\limits_{i=1}^{n}{\frac{1}{i}}=\log{n}+\gamma+O\hk*{\frac{1}{n}}, $ where $\gamma$ is the Euler's constant.
\begin{equation}\label{eq68}
    \sum_{j=1}^{N}{\frac{1}{2j+1}}<\frac{1}{2}\hk*{\log{(2N+1)}+\gamma -1 +O\hk*{\frac{1}{2N+1}}}.
\end{equation}
Note that $\frac{1}{2j-1}>\frac{1}{2j}$ for $1\leq j\leq c$, hence, it follows 
$$
\sum_{j=2}^{c}{\frac{1}{2j}}<\sum_{j=1}^{c-1}{\frac{1}{2j+1}}.
$$
Once again, adding $\sum\limits_{j=1}^{c-1}{\frac{1}{2j+1}}$ to both sides, we get
\begin{equation}\label{eq69}
    \sum_{j=1}^{c-1}{\frac{1}{2j+1}}>\frac{1}{2}\sum_{j=3}^{2c}{\frac{1}{j}}=\frac{1}{2}\hk*{\log{(2c)}+\gamma -\frac{3}{2} +O\hk*{\frac{1}{2c}}}.
\end{equation}
Combining the inequalities from (\ref{eq67}), (\ref{eq68}) and (\ref{eq69}), gives the upper bound 
\begin{align*}
    \log{P(c)}&\leq \frac{1}{2}\hk*{\log{\hk*{\frac{2N+1}{2c}}}+\frac{1}{2} +O\hk*{\frac{1}{2c}}+O\hk*{\frac{1}{2N+1}}}\\
    &\leq \log{\hk*{\hk*{\frac{2N+1}{2c}}^{\frac{1}{2}}}}+O(1) \text{, since }c, N\geq 1.
\end{align*}
Therefore, 
\begin{equation}
    P(c)\leq\hk*{\frac{2N+1}{2c}}^{\frac{1}{2}}O(e)=\frac{(2N+1)^{1/2}}{(2c)^{1/2}}O(1).
\end{equation}
Substitute this into (\ref{eq65}) by using (\ref{eq66}):
\begin{align*}
    s_{N+1} &\leq 4\sum_{c=1}^{N}{P(c)}\\
    &\leq 4\sum_{c=1}^{N}{\frac{(2N+1)^{1/2}}{(2c)^{1/2}}O(1)}\\
    &=O\hk*{N^{1/2}\sum_{c=1}^{N}{\frac{1}{c^{1/2}}}}.
\end{align*}
Shifting the integral gives 
$$\sum_{c=1}^{N}{\frac{1}{c^{1/2}}}\leq\int_{0}^{N}{\frac{1}{c^{1/2}}dt} = O\hk*{N^{1/2}} .$$
Therefore,
\begin{equation}
     s_{N+1}=O(N)
\end{equation}
\end{proof}
Substituting this into (\ref{eq63}) yields 
\begin{equation}
    (n+1)_{m_N}\leq n_{m_N-1}+O\hk*{N}\prod_{i=0}^{m_N-2}{\hk*{\frac{l_{\beta-i}}{l_{\beta-i}-1}}}.\label{eq72}
\end{equation}
Note that by equation (\ref{eq41}), $n_{m_N-1}=q_{m_N-1}(l_{\beta-m_N+1}-1)+r_{m_N-1}$. Recall that based on the definition of $m_N$, we have $q_{m_N-1}=k_N=N=\ceil*{\frac{\sqrt{l_n}}{\sqrt{\log (l_n)}}}$. Thus
\begin{equation}\label{eq73}
    n_{m_N-1}=\ceil*{\frac{\sqrt{l_n}}{\sqrt{\log (l_n)}}}(l_{\beta-m_N+1}-1)+r_{m_N-1}.
\end{equation}
Since $(n_i)$ is an increasing sequence with $n_{\beta}=l_n$, $n_{m_N-1}\leq n_\beta=l_n$. Substituting this into (\ref{eq73}) yields 
$$l_n/\ceil*{\frac{\sqrt{l_n}}{\sqrt{\log (l_n)}}}\geq l_{\beta-m_N+1}-1+r_{m_N-1}/\ceil*{\frac{\sqrt{l_n}}{\sqrt{\log (l_n)}}}$$
i.e.
$$l_n/\ceil*{\frac{\sqrt{l_n}}{\sqrt{\log (l_n)}}}+1\geq  l_{\beta-m_N+1} \text{, since } r_{m_N-1}\geq 0.$$
Therefore, for sufficiently large $n$, we have
\begin{equation}
    l_{\beta-m_N+1} = O\hk*{l_n^{1/2}\log^{1/2}{l_n}}.
\end{equation}
Recall that formula (\ref{eq3.11}) uses lucky numbers in decreasing order, hence using the counting function, there are less than $L\hk*{O\hk*{l_n^{1/2}\log^{1/2}{l_n}}}=O\hk*{L\hk*{l_n^{1/2}\log^{1/2}{l_n}}}$ steps left in the formula (or terms left in $(n_i)$).
Applying the lucky number theorem for sufficiently large $n$, it follows that the number of steps left is
\begin{align}\notag
    O\hk*{L\hk*{l_n^{1/2}\log^{1/2}{l_n}}}&=O\hk*{\frac{l_n^{1/2}\log^{1/2}{l_n}}{\log{\hk*{l_n^{1/2}\log^{1/2}{l_n}}}}}\\
    &=O\hk*{\frac{l_n^{1/2}\log^{1/2}{l_n}}{\log{l_n} + \log{\log{l_n}}}}\notag\\
    &=O\hk*{\frac{l_n^{1/2}}{\log^{1/2}{(l_n)}}}.
\end{align}
Start with equation (\ref{eq72}) and once again use the fact that $\floor*{x+y}\leq\floor{x}+y+1$, for $x,y\geq0$. (Apply this to all the remaining steps of formula (\ref{eq3.11}).)
\begin{align*}
    (n+1)_{m_N+1}&\leq n_{m_N}+O\hk*{N}\prod_{i=0}^{m_N-1}{\hk*{\frac{l_{\beta-i}}{l_{\beta-i}-1}}}+1\\
    &\leq n_{m_N}+\hk*{O\hk*{N}+1}\prod_{i=0}^{m_N-1}{\hk*{\frac{l_{\beta-i}}{l_{\beta-i}-1}}}.
\end{align*}
Repeat the process up to $(n+1)_{\beta+1}=l_{n+1}$ (for $O\hk*{\frac{l_n^{1/2}}{\log^{1/2}{(l_n)}}}$ steps) to get 
\begin{align*}
    (n+1)_{m_N+2}&\leq n_{m_N+1}+\hk*{O\hk*{N}+2}\prod_{i=0}^{m_N}{\hk*{\frac{l_{\beta-i}}{l_{\beta-i}-1}}}\\
    (n+1)_{m_N+3}&\leq n_{m_N+2}+\hk*{O\hk*{N}+3}\prod_{i=0}^{m_N+1}{\hk*{\frac{l_{\beta-i}}{l_{\beta-i}-1}}}\\
    & \; \;\vdots\\
    (n+1)_{\beta+1}&\leq n_{\beta}+\hk*{O\hk*{N}+O\hk*{\frac{l_n^{1/2}}{\log^{1/2}{(l_n)}}}}\prod_{i=0}^{\beta-1}{\hk*{\frac{l_{\beta-i}}{l_{\beta-i}-1}}}, \text{ for sufficiently large } n.
\end{align*}
Substituting $(n+1)_{\beta+1}=l_{n+1}$, $n_\beta=l_n$ and $N=\ceil*{\frac{l_n^{1/2}}{\log^{1/2}{(l_n)}}}=O\hk*{\frac{l_n^{1/2}}{\log^{1/2}{(l_n)}}}$, it follows that
\begin{equation}\label{eq76}
    l_{n+1}= l_{n}+O\hk*{\frac{l_n^{1/2}}{\log^{1/2}{(l_n)}}\prod_{i=0}^{\beta-1}{\hk*{\frac{l_{\beta-i}}{l_{\beta-i}-1}}}}, \text{ for sufficiently large } n.
\end{equation}
By Mertens' theorem \cite{r2}, it follows that for large enough $n$,
\begin{equation}\label{eq77}
    \prod_{i=0}^{\beta-1}{\hk*{\frac{l_{\beta-i}}{l_{\beta-i}-1}}}=\prod_{l\leq n}{\hk*{\frac{l}{l-1}}}=O(\log{n}).
\end{equation}
Clearly, since $n\leq l_n$, we have $\prod_{i=0}^{\beta-1}{\hk*{\frac{l_{\beta-i}}{l_{\beta-i}-1}}}=O(\log{(l_n)})$. Therefore,
$$
l_{n+1}=l_n+O\hk*{l_n^{1/2}\log^{1/2}{l_n}}
$$
i.e.
\begin{equation}
    l_{n+1}-l_n=O\hk*{l_n^{1/2}\log^{1/2}{l_n}}, \text{ for sufficiently large } n.
\end{equation}
This concludes the proof of the first part of the theorem.
Also note that by substituting equation (\ref{eq77}) into (\ref{eq76}) directly and then using the lucky number theorem for large enough $n$, we have
\begin{align}
    l_{n+1}&= l_{n}+O\hk*{\frac{l_n^{1/2}}{\log^{1/2}{(l_n)}}\log{n}}\notag\\
    &= l_{n}+O\hk*{\frac{n^{1/2}\log^{1/2}{n}}{\log^{1/2}{(n\log{n})}}\log{n}}\notag\\
    &= l_{n}+O\hk*{n^{1/2}\log{n}}, \text{ for sufficiently large } n.
\end{align}
This concludes the proof of the theorem.

As a direct consequence of this theorem, we have the following corollary.
\begin{cor}
    For sufficiently large $x$, there is always a lucky number between $x$ and $x+O\hk*{\sqrt{x\log{x}}}$.

    Or, there exists a constant $M\in\mathbb{N}$ and a constant $c$ such that for all $a\in \mathbb{N}$, if $a\geq M$, then there is always a lucky number between $a^2$ and $\hk*{a+c\log a}^2$.
\end{cor}

 For primes, Legendre's conjecture asks, is there always a prime between $a^2$ and $(a+1)^2$? This is still an open problem, and even solving the Riemann Hypothesis will not be strong enough to prove this result. An analogue of this conjecture would be is there always a lucky number between $a^2$ and $(a+1)^2$? If this is true, it would imply that $l_{n+1}-l_n=O(l_n^{1/2})$. Therefore, the result in the previous theorem is still slightly weaker than an analogue of Legendre's conjecture. However, it is the closest known result to it.
 Also note that these results could be extended to more general sequences of integers such as those defined in \cite{r3}.
\section{Conclusion}
 In this paper, we established multiple similarities between lucky numbers and prime numbers. We derived new formulas for lucky numbers, analogous to those obtained for other sieve-defined integer sequences. In addition, we proved a new result on the gaps between consecutive lucky numbers and proposed a fundamental theorem of arithmetic for lucky numbers. This theorem provides numerous new open questions and new potential areas for research. Many of the methods developed here can also be extended to more general sieve-generated sequences of integers. For example, some of these results, such as Theorem \ref{thm4.2} can be applied to sequences such as those studied in \cite{r3}.

Future research could focus on extending these ideas and addressing the questions that arise from Theorem \ref{thmA}. Several central problems remain unresolved \cite{r4}, including analogues of Legendre’s conjecture, Dirichlet’s theorem, the Goldbach conjecture, and the twin prime conjecture in the context of lucky numbers. It also remains unknown whether infinitely many lucky primes exist. We present many new problems such as finding the average order of the lucky arithmetical functions and a generalization of the lucky prime conjecture. Additionally, it is possible to ask a different version of an analogue of Dirichlet's theorem for lucky numbers: are there infinitely many lucky numbers in the arithmetic progression $a\circ n+b, \, n\in \mathbb{N}$ and some $a,b \in \mathbb{N}$? If there are, what values of $a,b$ will have this result? Similarly we may even ask are there infinitely many primes in the arithmetic progression $a\circ n+b, \, n\in \mathbb{N}$ and some $a,b \in \mathbb{N}$? Additionally, it is also possible to start studying Diophantine equations by using the new defined binary operations instead of multiplication or even combining multiplaction with the new binary operations.  
Clearly, the theory of lucky numbers is still in its early stages, with a wealth of conjectures and open problems to be explored. In particular, the fundamental theorem of arithmetic for lucky numbers suggests entirely new directions for research, offering a framework within which many further discoveries can be made.

Overall, the area of number theory that focuses on the study of lucky numbers and sequences of integers defined by sieves has great potential for future research.
\section*{Acknowledgments}
\begin{itemize}
    \item A special thanks to my girlfriend Lin\'e Crafford for her constant support and motivation throughout the year, and for always believing in me. 
\end{itemize}

\end{document}